\documentclass[11pt]{amsart}
\usepackage{amsfonts}
\usepackage{bbm}
\usepackage{amsfonts,amssymb,amsmath,amsthm}
\usepackage{url}
\usepackage{enumerate}
\usepackage{bbm}
\usepackage[all]{xy}
\usepackage[pdftex, colorlinks, citecolor=red, backref=page]{hyperref}
\usepackage{color,xcolor}

\newtheorem{theorem}{Theorem}[section]
\newtheorem{lemma}[theorem]{Lemma}
\newtheorem{definition}[theorem]{Definition}
\newtheorem{notion}[theorem]{}
\newtheorem{conjecture}[theorem]{Conjecture}

\newtheorem{proposition}[theorem]{Proposition}
\newtheorem{corollary}[theorem]{Corollary}
\theoremstyle{definition}
\newtheorem{remark}[theorem]{Remark}

\newtheorem{example}[theorem]{Example}

\author{Qingnan An}
\address{School of Mathematics and Statistics, Northeast Normal University, Changchun, {\rm130024}, China}
\email{qingnanan1024@outlook.com}

\author{Zhichao Liu}
\address{School of Mathematical Sciences,
Dalian University of Technology,
Dalian, {\rm116024}, China }
\email{lzc.12@outlook.com}


\begin{document}

\title[Classification problem for ${\rm C}^*$-algerbas] {Total Cuntz semigroup, Extension and Elliott Conjecture with Real rank zero}

\begin{abstract}
In this paper, we exhibit two unital, separable, nuclear ${\rm C}^*$-algebras of stable rank one and real rank zero with the same ordered scaled total K-theory, but they are not isomorphic with each other, which forms a counterexample to Elliott Classification Conjecture for real rank zero setting.
Thus, we introduce an additional normal condition and give a classification result in terms of total K-theory.
For the general setting, with a new invariant---total Cuntz semigroup \cite{AL}, we classify a large class of ${\rm C}^*$-algebras obtained from extensions. The total Cuntz semigroup,
which distinguish the algebras of our counterexample, could possibly classify all the ${\rm C}^*$-algebras of stable rank one and real rank zero.



\end{abstract}

\maketitle
\section*{Introduction}

The Elliott Conjecture asserts that all nuclear, separable ${\rm C}^*$-algebras are classified by suitable Elliott invariants. The original version of the conjecture was formulated in 1989 by Elliott in \cite{Ell}, in which he classified A$\mathbb{T}$-algebras of real rank zero. Since then, the classification program enjoyed tremendous success, the invariants had been extended and large class of ${\rm C}^*$-algebras had been classified.

For simple ${\rm C}^*$-algebras, Kirchberg and Phillips classified all the purely infinite, simple, separable, amenable ${\rm C}^*$-algebras satisfying the Universal Coefficient Theorem (see \cite{R} for an overview). For the stably finite simple case, the modified invariant consists of K-theory, traces, and the pairing between traces and K-theory. By now, it has been completed that all simple, separable, unital, nuclear, $\mathcal{Z}$-stable ${\rm C}^*$-algebras can be classified by the Elliott invariant provided the UCT holds (\cite{EGLN,GLN1,GLN2,TWW}), where $\mathcal{Z}$ is a simple, separable, nuclear and infinite dimensional ${\rm C}^*$-algebra with the same Elliott invariant as $\mathbb{C}$ and $\mathcal{Z}$-stability is a regularity property guarantees good structural behaviour \cite{GJSu}.

For real rank zero ${\rm C}^*$-algebras, the original Elliott invariant consisting of ordered ${\rm K}_*$-group and the scale (see \cite{Ell,EG}) has been modified to an extended form, in light of Gong's counterexample \cite{G, EGS}. Dadarlat and Loring also gave such an example in the class of AD algebras \cite{DL3}. The refined invariant is the ordered total $\mathrm{K}$-theory endowed with a certain order structure and acted upon by the natural coefficient transformations and by the Bockstein maps (see also \cite{DL2,Ei,RS,S}). From then on, the Elliott Conjecture for real rank zero setting had been modified to the following:

\begin{conjecture}\label{ell conj}
${({\rm\underline{K}}(A),{\rm\underline{K}}(A)_+,\Sigma A)}_{\rm \Lambda}$ is the complete invariant for separable nuclear ${\rm C}^*$-algebras of real rank zero and stable rank one.
\end{conjecture}

 We say a ${\rm C}^*$-algebra is an A$\mathcal{HD}$ algebra (see \cite[2.2-2.3]{GJL} and  \cite{GJL2}), if it is an inductive limit of finite direct sums of the form $M_n(\widetilde{\mathbb{I}}_p)$ and $PM_n(C(X))P$, where $$
\mathbb{I}_p=\{f\in M_p(C_0(0,1]):\,f(1)=\lambda\cdot1_p,\,1_p {\rm \,is\, the\, identity\, of}\, M_p\}
$$ is the Elliott-Thomsen dimension drop interval algebra
and $X$ is one of the following finite connected CW complexes: $\{pt\},~\mathbb{T},~[0, 1],~T_{II,k}.$ $P\in M_n(C(X))$ is a projection and $T_{II,k}$ is the 2-dimensional connected simplicial complex with $H^1(T_{II,k})=0$ and $H^2(T_{II,k})=\mathbb{Z}/k\mathbb{Z}$.

In 1997, Dadarlat and Gong \cite{DG} confirmed Conjecture \ref{ell conj} for  A${\mathcal{HD}}$-algebras of real rank zero. (In \cite[Section 1]{DG}, the A${\mathcal{HD}}$ class is called ASH(2) class, and it is proved in
\cite[Theorem 8.14]{DG} that this class includes all real rank zero AD algebras and AH algebras with slow dimension growth, though it is not known if it includes all real rank zero ASH algebras even with no dimension growth because the only subhomogeneous blocks included in \cite{DG} are dimension drop interval algebras.) Recently, \cite{AELL, ALZ} push forward a classification for certain ASH algebras with one-dimensional local spectra (it includes much more general subhomogeneous building blocks).
 It is believed that this conjecture can be completely verified for all ASH algebras of real rank zero, however, out of this setting, we will present a counterexample.

In this paper, we will use extension theory to construct two ${\rm C}^*$-algebras with the same ordered scaled total K-theory, but they are not isomorphic. Moreover, we point out that they have different total Cuntz semigroups. It has already been shown in \cite{AL} that total Cuntz semigroup may contain more information than total K-theory. 
So our work, on one hand, is to show the necessity  for an additional regularity assumption in Conjecture \ref{ell conj}, on the other hand, is devoted to find a suitable invariant for the general case. The appropriate condition might be ``K-pure'', which has close relation with weakly unperforated ${\rm K}_*$-group. We believe  that our consequences  will push forward the development of classification of nuclear ${\rm C}^*$-algebras of real rank zero.

We say a ${\rm C}^*$-algebra $A$ is K-$pure$, if for any ideal $I$ of $A$,  both the sequences
$$0\to  \mathrm{K}_j(I) \to\mathrm{K}_j(A)\to \mathrm{K}_j(A/I) \to 0,\quad j=0,1$$
are pure extensions of abelian groups. This concept is rooted in the extension theory of ${\rm C}^*$-algebras, which is an important tool for classification of ${\rm C}^*$-algebras together with K-theory and KK-theory. Many interesting ${\rm C}^*$-algebras has been constructed \cite{BD,DL0} through extensions. 

Our main results are as follows:
\begin{theorem}
  There exist two unital, separable, nuclear ${\rm C}^*$-algebras $E_1$ and $E_2$
 with stable rank one and real rank zero satisfying UCT and that
$$
(\underline{\mathrm{K}}(E_1),\underline{\mathrm{K}}(E_1)_+,[1_{E_1}])\cong
(\underline{\mathrm{K}}(E_2),\underline{\mathrm{K}}(E_2)_+,[1_{E_2}]),
$$
while $E_1\ncong E_2$. Moreover, both $E_1$ and $E_2$ are not {\rm K}-pure.
\end{theorem}

These two algebras $E_1, E_2$ are obtained from non K-pure extensions of ${\rm C}^*$-algebras classified in \cite{DG} and we will prove that
$$
\underline{\rm Cu}_u(E_1)\ncong_L \underline{\rm Cu}_u(E_2),
$$
where $\underline{\rm Cu}_u$ is a functor recovering $(\underline{\mathrm{K}},\underline{\mathrm{K}}_+,[1])$ (see \cite{AL}). This  result disproves Conjecture \ref{ell conj} and also answers \cite[Question 5.14]{AL} raised by the authors.

We will also present a characterization of  K-pure ${\rm C}^*$-algebras and establish equivalent formulations for K-pure extension. 
Classification theorems are given for both K-pure extension case in terms of total K-theory and general extension case in terms of total Cuntz semigroup. 


\begin{theorem}
Let $A_1, A_2, B_1,B_2$ be A${\mathcal{HD}}$-algebras of real rank zero, and $A_1, A_2$ are unital simple, $B_1, B_2$ are stable. Given two unital essential {\rm K}-pure extensions
$$
0\to B_i \to E_i\to A_i\to 0,\quad i=1,2.
$$
We have $E_1\cong E_2$ iff
$$(\underline{\mathrm{K}}(E_1),\underline{\mathrm{K}}(E_1)_+,[1_{E_1}])\cong
(\underline{\mathrm{K}}(E_2),\underline{\mathrm{K}}(E_2)_+,[1_{E_2}]).$$
\end{theorem}
\begin{theorem}
Let $A_1, A_2, B_1,B_2$ be A${\mathcal{HD}}$-algebras of real rank zero, and $A_1, A_2$ are unital simple, $B_1, B_2$ are stable.
Given two unital essential extensions (not necessarily $\mathrm{K}$-pure) with trivial index maps
$$
0\to B_i \to E_i\to A_i\to 0,\quad i=1,2.
$$
We have $E_1\cong E_2$ iff
$$\underline{\mathrm{Cu}}_u(E_1)\cong_L \underline{\mathrm{Cu}}_u(E_2),$$
where $\cong_L$ is the latticed ${\underline{\rm Cu}}$-isomorphism.
\end{theorem}

The above results present the further evidence that the total K-theory might be complete for the  ``well-behaved" ${\rm C}^*$-algebras and  it is necessary to consider the  new ingredients in the general classification theory. It can be expected that the Elliott conjecture for real rank zero ${\rm C}^*$-algebras might be accomplished through the following versions:

\begin{conjecture}
{\bf (Strong)} ${({\rm\underline{\mathrm{K}}}(A),{\rm\underline{K}}(A)_+,\Sigma A)}_{\rm \Lambda}$ is the complete invariant for separable, nuclear,  {\rm K}-pure ${\rm C}^*$-algebras of real rank zero and stable rank one.
\end{conjecture}
\begin{conjecture}
{\bf (Weak) }$\underline{\mathrm{Cu}}_u(A)$ is the complete invariant for unital, separable, nuclear ${\rm C}^*$-algebras of real rank zero and stable rank one.
\end{conjecture}

This paper is organized as follows. In Section 1, we list some preliminaries for total $\mathrm{K}$-theory with the Dadarlat-Gong order,
extension theory and also some basic results.
In Section 2, we begin with an example of two non isomorphic algebras $B_1,B_2$ raised in \cite{DL3}
and take two (non K-pure) unital essential extensions with trivial index maps 
$$
e_i: \quad 0\to B_i\to E_i\to A_\chi, \quad i=1,2,
$$
where $A_\chi$ is an algebra classified in \cite{EG}.
We will check that $E_1,E_2$ have the same  ordered scaled total K-theory, but they are not isomorphic.
In Section 3, we will prove some properties for K-pure  extensions and K-pure ${\rm C}^*$-algebras.
In Section 4 and Section 5, we will classify the algebras obtained from unital K-pure  extensions in terms of total K-theory and the algebras from general extensions in terms of total Cuntz semigroup introduced in \cite{AL}.

\section{Preliminaries}

\begin{notion}\rm
  Let $A$ be a unital $\mathrm{C}^*$-algebra. $A$ is said to have stable rank one, written $sr(A)=1$, if the set of invertible elements of $A$ is dense. $A$ is said to have real rank zero, written $rr(A)=0$, if the set of invertible self-adjoint elements is dense in the set $A_{sa}$ of self-adjoint elements of $A$. If $A$ is not unital, let us denote the minimal unitization of $A$ by $\widetilde{A}$. A non-unital $\mathrm{C}^*$-algebra is said to have stable rank one (or real rank zero) if its unitization has stable rank one (or real rank zero).

  Let $p$ and $q$ be two projections in $A$. One says that $p$ is $Murray$--$von$ $Neumann$ $equivalent$ to $q$ in $A$ and writes $p\sim q$ if there exists $x\in A$ such that $x^*x=p$ and $xx^*=q$. We will write $p\preceq q$ if $p$ is equivalent to some subprojection of $q$. The class of a projection $p$ in $\mathrm{K}_0(A)$ will be denoted by $[p]$ and  we say $[p]\leq [q]$, if $p\preceq q$.

$A$ is said to have cancellation of projections, if for any projections $p,q,e,f\in A$ with $pe=0$, $qf=0$, $e\sim f$, and $p+e\sim q+f$, then $p\sim q$. $A$ has cancellation of projections if and only if $p\sim q$ implies that there exists a unitary $u\in \widetilde{A}$ such that $u^*pu=q$ (see \cite[3.1]{L}). Every unital $\mathrm{C}^*$-algebra of stable rank one has cancellation of projections, hence, any two projections with the same ${\rm K}$-theory generate the same ideal; frequently, we will also use $I_{[p]}$ to represent the ideal $I_p$ generated by $p$ in $A$.
\end{notion}

\begin{notion}\label{def k-total}\rm
 {\bf (The total K-theory)} (\cite[Section 4]{DG}) For $n\geq 2$, the mod-$n$ K-theory groups are defined by
$$\mathrm{K}_* (A;\mathbb{Z}_n)=\mathrm{K}_*(A\otimes C_0(W_n)),$$
where $\mathbb{Z}_n:=\mathbb{Z}/n\mathbb{Z}$, $W_n$ denotes the Moore space obtained by attaching the unit disk to the circle by a degree $n$-map, such as $f_n: \mathbb{T}\rightarrow \mathbb{T}$, $e^{{\rm i}\theta}\mapsto e^{{\rm i}n\theta}$. The ${\rm C}^*$-algebra $C_0(W_n)$ of continuous functions vanishing at the base point is isomorphic to the mapping cone of the canonical map of degree $n$ from $C(\mathbb{T})$ to itself.

 In the setting of \cite{S} (see also \cite{Kar,Cu mod p}), the mod-$n$ ${\rm K}$-theory groups are defined by
$$
{\rm K}_{j}(A ; \mathbb{Z}_n)={\rm K}_{j}\left(A \otimes C_{0}\left(W_{n}\right)\right),\,\,\,\,j=0,1.
$$

Let ${\rm K}_{*}(A ; \mathbb{Z}_n)={\rm K}_{0}(A ; \mathbb{Z}_n) \oplus {\rm K}_{1}(A ; \mathbb{Z}_n)$. For $n=0$, we set ${\rm K}_{*}(A ; \mathbb{Z}_n)=$ ${\rm K}_{*}(A)$ and for $n=1, {\rm K}_{*}(A ; \mathbb{Z}_n)=0$.

For a $\mathrm{C}^{*}$-algebra $A$, one defines the total K-theory of $A$ by
$$
\underline{{\rm K}}(A)=\bigoplus_{n=0}^{\infty} {\rm K}_{*}(A ; \mathbb{Z}_n) .
$$
It is a $\mathbb{Z}_2 \times \mathbb{Z}^{+}$graded group. It was shown in \cite{S} that the coefficient maps
$$
\begin{gathered}
\rho: \mathbb{Z} \rightarrow \mathbb{Z}_n, \quad \rho(1)=[1], \\
\kappa_{mn, m}: \mathbb{Z}_m \rightarrow \mathbb{Z}_{mn}, \quad \kappa_{m n, m}[1]=n[1], \\
\kappa_{n, m n}: \mathbb{Z}_{mn} \rightarrow \mathbb{Z}_n, \quad \kappa_{n, m n}[1]=[1],
\end{gathered}
$$
induce natural transformations
$$
\rho_{n}^{j}: {\rm K}_{j}(A) \rightarrow {\rm K}_{j}(A ; \mathbb{Z}_n),
$$
$$
\kappa_{m n, m}^{j}: {\rm K}_{j}(A ; \mathbb{Z}_m) \rightarrow {\rm K}_{j}(A ; \mathbb{Z}_{mn}),
$$
$$
\kappa_{n, m n}^{j}: {\rm K}_{j}(A ; \mathbb{Z}_{mn}) \rightarrow {\rm K}_{j}(A ; \mathbb{Z}_n) .
$$
The Bockstein operation
$$
\beta_{n}^{j}: {\rm K}_{j}(A ; \mathbb{Z}_n) \rightarrow {\rm K}_{j+1}(A)
$$
appears in the six-term exact sequence
$$
{\rm K}_{j}(A) \stackrel{\times n}{\longrightarrow} {\rm K}_{j}(A) \stackrel{\rho_{n}^{j}}{\longrightarrow} {\rm K}_{j}(A ; \mathbb{Z}_n) \stackrel{\beta_{n}^{j}}{\longrightarrow} {\rm K}_{j+1}(A) \stackrel{\times n}{\longrightarrow} {\rm K}_{j+1}(A)
$$
induced by the cofibre sequence
$$
A \otimes S C_{0}\left(\mathbb{T}\right) \longrightarrow A \otimes C_{0}\left(W_{n}\right) \stackrel{\beta}{\longrightarrow} A \otimes C_{0}\left(\mathbb{T}\right) \stackrel{n}{\longrightarrow} A \otimes C_{0}\left(\mathbb{T}\right),
$$
where $S C_{0}\left(\mathbb{T}\right)$ is the suspension algebra of $C_{0}\left(\mathbb{T}\right)$.

There is a second six-term exact sequence involving the Bockstein operations. This is induced by a cofibre sequence
$$
A \otimes S C_{0}\left(W_{n}\right) \longrightarrow A \otimes C_{0}\left(W_{m}\right) \longrightarrow A \otimes C_{0}\left(W_{m n}\right) \longrightarrow A \otimes C_{0}\left(W_{n}\right)
$$
and takes the form:
$$
{\rm K}_{j+1}(A ; \mathbb{Z}_n) \stackrel{\beta_{m, n}^{j+1}}{\longrightarrow} {\rm K}_{j}(A ; \mathbb{Z}_m) \stackrel{\kappa_{m n, m}^{j}}{\longrightarrow} {\rm K}_{j}(A ; \mathbb{Z}_{mn}) \stackrel{\kappa_{n, m n}^{j}}{\longrightarrow} {\rm K}_{j}(A ; \mathbb{Z}_n),
$$
where $\beta_{m, n}^{j}=\rho_{m}^{j+1} \circ \beta_{n}^{j}$.

The collection of all the transformations $\rho, \beta, \kappa$ and their compositions is denoted by $\Lambda$. $\Lambda$ can be regarded as the set of morphisms in a category whose objects are the elements of $\mathbb{Z}_2 \times \mathbb{Z}^{+}$. Abusing the terminology, $\Lambda$ will be called the category of Bockstein operations. Via the Bockstein operations, $\underline{{\rm K}}(A)$ becomes a $\Lambda$-module. It is natural to consider the group $\operatorname{Hom}_{\Lambda}(\underline{{\rm K}}(A), \underline{{\rm K}}(B))$ consisting of all $\mathbb{Z}_2 \times \mathbb{Z}^{+}$ graded group morphisms which are $\Lambda$-linear, i.e. preserve the action of the category $\Lambda$.

The Kasparov product induces a map
$$
\gamma_{n}^{j}: {\rm K K}(A, B) \rightarrow \operatorname{Hom}\left({\rm K}_{j}(A ; \mathbb{Z}_n), {\rm K}_{j}(B ; \mathbb{Z}_n)\right).
$$
Then $\gamma_{n}=\left(\gamma_{n}^{0}, \gamma_{n}^{1}\right)$ will be a map
$$
\gamma_{n}: {\rm K K}(A, B) \rightarrow \operatorname{Hom}\left({\rm K}_{*}(A ; \mathbb{Z}_n), {\rm K}_{*}(B ; \mathbb{Z}_n)\right).
$$
Note that if $n=0$, then ${\rm K}_{*}(A, \mathbb{Z}_n)={\rm K}_{*}(A)$ and the map $\gamma_{0}$ is the same as the map $\gamma$ from the Universal Coefficient Theorem (UCT) of Rosenberg and Schochet \cite{RS}. We assemble the sequence $\left(\gamma_{n}\right)$ into a map $\Gamma$. Since the Bockstein operations are induced by multiplication with suitable KK elements and since the Kasparov product is associative, we obtain a map
$$
\Gamma: {\rm K K}(A, B) \rightarrow \operatorname{Hom}_{\Lambda}(\underline{{\rm K}}(A), \underline{{\rm K}}(B)) \text {. }
$$
For the sake of simplicity, if $\alpha \in {\rm K K}(A, B)$, then $\Gamma(\alpha)$ will be often denoted by $\underline{\alpha}$.

For any $n\in \mathbb{N},\,j=0,1$, if $\psi:\, A\to B$ is a homomorphism, we will denote
 $\underline{\rm K}(\psi):\, \underline{{\rm K}}(A)\to\underline{{\rm K}}(B)$ and ${\rm K}_j (\psi;\mathbb{Z}_n):\, {{\rm K}}_j(A;\mathbb{Z}_n)\to{{\rm K}}_j(B;\mathbb{Z}_n)$ to be the induced maps.

 If $\beta:\,\underline{{\rm K}}(A)\to\underline{{\rm K}}(B)$ is a graded map, we will denote
$$
\beta_n^j:\,{\rm K}_j(A;\mathbb{Z}_n)\to{\rm K}_j(B;\mathbb{Z}_n),\,\,n\in \mathbb{N},\,j=0,1
$$
to be the restriction maps.

In fact, as
pointed out by Cuntz and Schochet \cite[Theorem 6.4]{S}, one can also define ${\rm K}_{*}(A ; \mathbb{Z}_n)$
using the (non-commutative!) Cuntz algebra $\mathcal{O}_{n+1}$ (see \cite{Cu alg, Cu ann}) in place of $W_n$.
\end{notion}

\begin{notion}\label{dg order}{\bf (Dadarlat-Gong order) ~}\rm
 Assume that $A$ is a separable C*-algebra of stable rank one and
$\mathcal{K}$  is the compact operators on a separable infinite-dimensional Hilbert space.
For any $[e]\in \mathrm{K}_0^+(A)$, denote
$I_e$  the ideal of $A\otimes\mathcal{K}$ generated by $e$ and
denote $\underline{\mathrm{K}}({I_e}\,|\,A)$ to be the image of $\underline{\mathrm{K}}({I_e})$ in $\underline{\mathrm{K}}(A)$, i.e.,
$$\underline{\mathrm{K}}({I_e}\,|\,A)=:
\underline{\mathrm{K}}(\iota_{I_e})(\underline{\mathrm{K}}({I_e}))\subset \underline{\mathrm{K}}(A),$$
where $\iota_{I_e}:\,I_e\to A\otimes\mathcal{K} $ is the natural embedding map. The following is a positive cone for total K-theory of $A$ (\cite[Definition 4.6]{DG}):
$$\textstyle
\underline{\mathrm{K}}(A)_+=\{([e],\mathfrak{u},
\bigoplus\limits_{n=1}^{\infty}(\mathfrak{s}_{n,0},\mathfrak{s}_{n,1})):[e]\in \mathrm{K}_0^+(A),([e],\mathfrak{u},\bigoplus\limits_{n=1}^{\infty}
(\mathfrak{s}_{n,0},\mathfrak{s}_{n,1}))\in \underline{\mathrm{K}}({I_e}\,|\, A)\},
$$
where $\mathfrak{u}\in {\rm K}_1(A)$, $\mathfrak{s}_{n,0}\in {\rm K}_0(A;\mathbb{Z}_n)$ and  $\mathfrak{s}_{n,1}\in {\rm K}_1(A;\mathbb{Z}_n)$ are equivalent classes. We will call this order structure by Dadarlat-Gong order.

In particular, we may also denote
$${\mathrm{K}_j}({I_e}\,|\,A;\mathbb{Z}_n)=:
{\mathrm{K}_j}(\iota_{I_e};\mathbb{Z}_n)({\mathrm{K}_j}({I_e};\mathbb{Z}_n))\subset {\mathrm{K}_j}(A;\mathbb{Z}_n),\quad j=0,1,$$
$$\mathrm{K}_*^+(A)=:\mathrm{K}_*(A)\cap \underline{\mathrm{K}}(A)_+,$$
where $\mathrm{K}_*(A)$ is identified with its natural image in $\underline{\mathrm{K}}(A)$.
\end{notion}

Now we have the following proposition immediately.
\begin{proposition}\label{gong ideal preserving}
Let $A, B$ be ${\rm C}^*$-algebras of stable rank one and
let $\phi:\, {\rm \underline{K}}(A)\to {\rm \underline{K}}(B)$ be a morphism of $\mathbb{Z}_2\times \mathbb{Z}^+$-graded
groups. The following statements are equivalent:

(i) $\phi({\rm \underline{K}}(A)_+)\subset{\rm \underline{K}}(B)_+$;

(ii) For any $(x,\overline{x})$ with $x\in {\rm {K}}_0^+(A)$
and
$
\overline{x}\in {\rm {K}}_1(I_x)\times\bigoplus_{n= 1}^{\infty}{\rm{K}}_*(I_x;\mathbb{Z}_n),
$
there exists $(y,\overline{y})$ with $y\in {\rm {K}}_0^+(B)$
and $
\overline{y}\in {\rm {K}}_1(I_y)\times\bigoplus_{n= 1}^{\infty}{\rm{K}}_*(I_y;\mathbb{Z}_n)
$
such that
$$
\phi\circ\underline{\rm K}(\iota_{I_x}) (x,\overline{x})=\underline{\rm K}(\iota_{I_y})(y,\overline{y}),
$$
where $I_x$, $I_y$ are the ideals of $A\otimes \mathcal{K}$, $B\otimes \mathcal{K}$ generated by $x, y$, respectively,
and $\iota_{I_x}:\,I_x\to A\otimes \mathcal{K}$,
$\iota_{I_y}:\,I_y\to B\otimes \mathcal{K}$ are the natural injective maps.
\end{proposition}

\begin{example}
Denote $\mathcal{Q}$  the  UHF algebra  whose $\mathrm{K}_0$-group is
$(\mathbb{Q},\mathbb{Q}_+,1)$
and recall that  $\mathcal{K}$ is the compact operators on a separable infinite-dimensional Hilbert space.

Then
$$(\underline{\mathrm{K}}(\mathcal{Q}),\underline{\mathrm{K}}(\mathcal{Q})_+,[1_\mathcal{Q}])\cong
(\mathbb{Q},\mathbb{Q}_+,1);
$$
$$
\underline{\mathrm{K}}(\mathcal{K})\cong
\mathbb{Z}\oplus \mathbb{Z}_2\oplus\mathbb{Z}_3\oplus\mathbb{Z}_4\oplus\cdots
$$
and $$\underline{\mathrm{\mathrm{K}}}(\mathcal{K})_+=\{(x,x_2,x_3,x_4,\cdots)\in\underline{\mathrm{K}}(\mathcal{K})\mid x>0\}\cup \{(0,0,0,0,\cdots)\},$$
where $\mathrm{K}_0(\mathcal{K}; \mathbb{Z}_n)\cong\mathbb{Z}_n$ and $\mathrm{K}_1(\mathcal{K}; \mathbb{Z}_n)\cong0.$
\end{example}

\begin{theorem}{\rm (\cite[Porposition 4.8--4.9]{DG})}\label{ordertotal}
Suppose that $A$ is of stable rank one and has an approximate unit $(e_n)$ consisting
of projections. Then

(i) $\underline{\mathrm{K}}(A)=\underline{\mathrm{K}}(A)_+-\underline{\mathrm{K}}(A)_+$;

(ii) $\underline{\mathrm{K}}(A)_+\cap\{-\underline{\mathrm{K}}(A)_+\} = \{0\}$, and hence,
$(\underline{\mathrm{K}}(A),\underline{\mathrm{K}}(A)_+)$ is an ordered group;

(iii) For any $x\in \underline{\mathrm{K}}(A)$, there are positive integers $k$, $n$ such that $k[e_n]+x \in \underline{\mathrm{K}}(A)_+$.
\end{theorem}
Note that separability and  real rank zero imply the existence of an approximate unit consisting
of projections. In the above theorem, if we replace $\underline{\mathrm{K}}$ by
${\mathrm{K}}_0$ or ${\mathrm{K}}_*$, the statement is still true.

\begin{definition}\label{group pure}\rm
   An extension of abelian groups
   $$0\rightarrow  K \xrightarrow{\iota} G \xrightarrow{\pi} H\rightarrow 0$$
is called $pure$, if $\iota(K)$ is a pure subgroup of $G$, i.e.,
$
n\times \iota(K)=\iota(K)\cap (n\times G)$   for every $n\in\mathbb{N}$.
\end{definition}
\begin{definition}\rm
Let $A$ and $B$ be ${\rm C}^*$-algebras. An extension $e$ of $A$ by $B$ is a short exact sequence of ${\rm C}^*$-algebras:
 $$
e:\quad 0 \rightarrow B \rightarrow E \rightarrow A \rightarrow 0.
$$
We say $e$ is $\mathrm{K}$-$pure$,
if both the sequences
$$0\to  \mathrm{K}_j(B ) \to\mathrm{K}_j(E)\to \mathrm{K}_j(A ) \to 0,\quad j=0,1$$
are pure extensions of abelian groups.
\end{definition}

\begin{definition}\label{def kpure}\rm
Given an ideal $I$ of $A$, we say that $I$ is  $\mathrm{K}$-$pure$ in $A$, if the extension
$
0\to I\to A\to A/I\to 0
$
is $\mathrm{K}$-pure. We say a ${\rm C}^*$-algebra $A$ is $\mathrm{K}$-$pure$, if all the ideals of $A$ are $\mathrm{K}$-pure in $A$.
\end{definition}




It is obviously that if $A, B$ are K-pure, then $A\oplus B$ and $A\otimes\mathcal{K}$ are K-pure. By \cite[Proposition 4.4]{DE},  all  A$\mathcal{HD}$ algebras of real rank zero are K-pure.

\begin{notion}\rm
Given abelian groups $H,K$ and two extensions
$$
e_i\,\,:\,\,0\to K\to G_i\to H\to 0,\quad i=1,2,
$$
we say $e_1$, $e_2$ are equivalent, if there is a homomorphism  $\alpha$ making the following diagram commute:
$$
\xymatrixcolsep{2pc}
\xymatrix{
{\,\,0\,\,} \ar[r]^-{}
& {\,\,K\,\,} \ar[d]_-{{\rm id}} \ar[r]^-{}
& {\,\,G_1\,\,} \ar[d]_-{\alpha} \ar[r]^-{}
& {\,\,H\,\,} \ar[d]_-{{\rm id}} \ar[r]^-{}
& {\,\,0\,\,} \\
{\,\,0\,\,} \ar[r]^-{}
& {\,\,K\,\,} \ar[r]_-{}
& {\,\,G_2 \,\,} \ar[r]_-{}
& {\,\,H \,\,} \ar[r]_-{}
& {\,\,0\,\,}.}
$$
Denote $\mathrm{Ext}(H,K)$ the set of all the equivalent classes of extensions of $H$ by $K$. It is well-known that $\mathrm{Ext}(H,K)$ forms an abelian group. 

Let $h_0 \in H$. We consider the following
extension of $H$ by $K$ with base point $h_0$,
$$
e:\quad 0\to K\to (G, g_0)\xrightarrow{\psi} (H, h_0)\to 0,
$$
where $g_0 \in G$ and $\psi(g_0) = h_0$.
We say $e$ is trivial if there exists a homomorphism
$\lambda : H \to G$ such that $\psi\circ \lambda={\rm id}_{H}$ and $\lambda(h_0) = g_0$.

Suppose we have extensions
$$e_i:\quad
0\to K\to (G_i, g_i)\xrightarrow{\psi} (H, h_0)\to 0,\quad i = 1, 2$$
 with $\psi(g_i) = h_0$. We say  $e_1$ and $e_2$ are equivalent if there is a
homomorphism $\phi : G_1 \to G_2$ with $\phi(g_1) = g_2$ such that the following diagram commutes
$$
\xymatrixcolsep{2pc}
\xymatrix{
{\,\,0\,\,} \ar[r]^-{}
& {\,\,K\,\,} \ar[d]_-{{\rm id}} \ar[r]^-{}
& {\,\,(G_1, g_1)\,\,} \ar[d]_-{\phi} \ar[r]^-{}
& {\,\,(H, h_0)\,\,} \ar[d]_-{{\rm id}} \ar[r]^-{}
& {\,\,0\,\,} \\
{\,\,0\,\,} \ar[r]^-{}
& {\,\,K\,\,} \ar[r]_-{}
& {\,\,(G_2, g_2) \,\,} \ar[r]_-{}
& {\,\,(H, h_0) \,\,} \ar[r]_-{}
& {\,\,0\,\,}.}
$$

Let $\mathrm{Ext}((H, h_0), K)$ be the set of equivalence classes of all extensions of $H$ by $K$ with base point $h_0$. As the usual group-theoretic construction of $\mathrm{Ext}(H, K)$, there exists an analogue construction on $\mathrm{Ext}((H, h_0), K)$ such that it is an abelian group.
One can check that there is a short exact sequence of groups
$$0 \to K/\{f(h_0)|f \in {\rm Hom}(H, K)\} \to \mathrm{Ext}((H, h_0), K)\to \mathrm{Ext}(H, K) \to 0.$$

\end{notion}

The following proposition is well-known, see \cite[Proposition 4]{LR}.
\begin{proposition}\label{lin inj}
Let
$$
0 \longrightarrow B \longrightarrow E \longrightarrow A \longrightarrow 0
$$
be an extension  of $\mathrm{C}^{*}$-algebras. Let $\delta_{j}: \mathrm{K}_{j}(A) \rightarrow \mathrm{K}_{1-j}(B)$ for $j=0,1$, be the index maps of the sequence.

(i) Assume that $A$ and $B$ have real rank zero. Then the following three conditions are equivalent :

(a) $\delta_{0} \equiv 0$,

(b) $rr(E)=0$,

(c) all projections in $A$ are images of projections in $E$.

(ii) Assume that $A$ and $B$ have stable rank one. Then the following are equivalent:

(a) $\delta_{1} \equiv 0$,

(b) $sr(E)=1$.

If, in addition, $E$ (and $A$ ) are unital, then (a) and (b) in (ii) are equivalent to

(c) all unitaries in $A$ are images of unitaries in $E$.

\end{proposition}
\begin{definition}\label{trivial index def}\rm
We say an extension $e:0\to B\to E\to A\to 0$ has $trivial$ $index$ $maps$, if both $\delta_0$ and $\delta_1$ are the zero maps.

Thus, suppose both $A,B$ are of stable rank one and real
rank zero, by Proposition \ref{lin inj}, we point out that $E$ has stable rank one and real
rank zero if and only if $e$ has trivial index maps.
\end{definition}
\begin{definition}\rm
Let $A$, $B$ be ${\rm C}^*$-algebras and
$$e\,:\, 0 \to B \to E \to A \to 0 $$
be an extension of $A$ by $B$ with Busby invariant $\tau:\,A\to M(B)/B$, where $M(B)$ is the multiplier algebra of $B$.
We say extension $e$ is $essential$, if $\tau$ is injective.
When $A$ is unital, we say an extension $e$ is $unital$, if $\tau$ is unital.
\end{definition}
\begin{proposition}
Assume that $A$ is unital simple and $B$ is stable with $B\neq 0$. Then any unital extension $e$ of $A$ by $B$ is always essential.

\end{proposition}
\begin{proof}
For any unital extension $e$ of $A$ by $B$ with Busby invariant $\tau$,
since $\tau$ is a unital homomorphism from $A$ to $M(B)/B$ and $A$ has no proper ideals,
the Busby invariant $\tau$ must be injective, which concludes the proof.

\end{proof}

\begin{definition}\rm
Let
$e_i:0 \to B \to E_i \to A \to 0 $
be two
extensions of $A$ by $B$ with Busby invariants $\tau_i$ for $i = 1, 2$.  $e_1$ and $e_2$ are called $strongly$ $unitarily$ $equivalent$, denoted by $e_1
\sim_s e_2$,
if there exists a unitary $u \in M(B)$ such that $\tau_2(a) = \pi(u)\tau_1(a)\pi(u)^*$ for all $a \in A$,
where $\pi:\,M(B)\to M(B)/B.$


\end{definition}
\begin{proposition}\label{strong tui cong}
Given two 
extensions $e_1$ and $e_2$ with $e_1\sim_s e_2$, we have $E_1\cong E_2$.
\end{proposition}
\begin{proof}
Via the pull back construction, we have
$$
E_i= \{(a,b)\in A\times M(B)\mid\, \tau_i(a)=\pi(b)\},\quad i=1,2.
$$
$e_1\sim_s e_2$ means that there exists a unitary $u \in M(B)$ such that $\tau_2(a) = \pi(u)\tau_1(a)\pi(u)^*$ for all $a \in A$.
Then
the map $\psi:\,E_1\to E_2$ with
$$
\psi(a,b)=(a,ubu^*)
$$
is an isomorphism.

\end{proof}


\begin{notion}\rm
Let $A$ and $B$ be ${\rm C}^*$-algebras. When $A$ is unital, we set
$$\mathrm{Ext}_{[1]}(\mathrm{K}_*(A), \mathrm{K}_*(B)) = \mathrm{Ext}((\mathrm{K}_0(A), [1_A]), \mathrm{K}_0(B)) \oplus \mathrm{Ext}(\mathrm{K}_1(A), \mathrm{K}_1(B)).$$
\end{notion}
Denote $\mathcal{N}$ the ``bootstrap" class of ${\rm C}^*$-algebras in \cite{RS}.
We list various versions of ``UCT" results as follows:

\begin{theorem}{\rm (}\cite{RS},  {\bf UCT} {\rm )}\label{UCT}
Let $A, B$ be separable ${\rm C}^*$-algebras. Suppose that $A\in \mathcal{N}$. Then there is a short exact sequence
$$
0 \to \mathrm{Ext}(\mathrm{K}_*(A), \mathrm{K}_*(B)) \xrightarrow{\delta} \mathrm{K}\mathrm{K}^*(A, B) \xrightarrow{\gamma} {\rm Hom}(\mathrm{K}_*(A), \mathrm{K}_*(B))\to 0,
$$
where the map $\gamma$ has degree $0$ and $\delta$ has degree $1$.
\end{theorem}
\begin{theorem}[Proposition 7.3 and Corollary 7.5 in \cite{RS}]\label{kk-equivalence}
Let $A_1$ and $A_2$ be ${\rm C}^*$-algebras in $\mathcal{N}$ with the same (abstract) $\mathrm{K}$-groups, that is,
we have a graded isomorphism $\rho:\,\mathrm{K}_*(A_1)\to \mathrm{K}_*(A_2)$, then there exists  a $\mathrm{K}\mathrm{K}$-equivalence $\lambda\in \mathrm{K}\mathrm{K}(A_1, A_2)$
have the property that $\gamma(\lambda)=\lambda_*=\rho\in {\rm Hom}(\mathrm{K}_*(A_1), \mathrm{K}_*(A_2))$.
\end{theorem}
\begin{theorem}{\rm (}\cite{DL2}, \cite[Theorem 4.2]{DG}, {\bf UMCT} {\rm )}\label{UMCT}
Let $A, B$ be ${\rm C}^*$-algebras. Suppose that $A\in \mathcal{N}$
and $B$ is $\sigma$-unital. Then there is a short exact sequence
$$
0 \to {\rm Pext}(\mathrm{K}_*(A), \mathrm{K}_*(B)) \to \mathrm{K}\mathrm{K}(A, B) \to {\rm Hom}_\Lambda(\underline{\mathrm{K}}(A), \underline{\mathrm{K}}(B))\to 0,
$$
which is natural in each variable.
\end{theorem}
For an abelian group $G$, we denote its subgroup consisting of all the torsion elements of $G$   by $\operatorname{tor}(G)$.
\begin{definition}\rm {\rm (}\cite{Ell0},\cite[Lemma 8.1]{Goo2}, \cite[Definition 1.2.7]{EG}{\rm )}\label{0weakly un}
An ordered abelian group $\left(G, G_{+}\right)$ is called weakly unperforated if it satisfies both of the following conditions:

(i) $G/\operatorname{tor}(G)$ is unperforated, i.e, in the quotient ordered group, $ng \geq 0$ with
$ n\in \mathbb{N}^+$  implies $g\geq0$;

(ii) Given $g\in G_+$, $t\in  \operatorname{tor}(G)$, $n\in \mathbb{N}^+$ and $m\in \mathbb{Z}$ with $ng + mt\in G_+$, then
$t = t' + t''$ for some $t', t''\in\operatorname{tor}(G)$ such that  $mt' = 0$ and $g+ t''\in G_+$.
\end{definition}

We will also need the following result obtained from \cite[Theorem 4.9]{W}, one can also see analogous versions under the circumstance in \cite{Sk1,Sk2,W0}.
\begin{theorem}[Theorem 3.5 of \cite{AL2}]\label{strong wei}
Let $A,B$ be nuclear separable ${\rm C}^*$-algebras of stable rank one and real rank zero with $A\in \mathcal{N}$. Assume that  $A$ is unital simple,  $B$ is stable and $({\rm K}_0(B),{\rm K}_0^+(B))$  is weakly unperforated.
Let ${\rm Text}_{s}^u(A,B)$ be  the set of strongly unitary equivalence classes of all the unital extensions of $A$ by $B$ with trivial index maps. Then  we have
$$
\mathrm{Ext}_{[1]}(\mathrm{K}_*(A), \mathrm{K}_*(B))\cong {\rm Text}_{s}^u(A,B).
$$
\end{theorem}
Suppose $A,B$ are ${\rm C}^*$-algebras satisfy the above conditions  and $$\varepsilon\in \mathrm{Ext}_{[1]}(\mathrm{K}_*(A), \mathrm{K}_*(B))$$ is the equivalent class of the extension
$$
0\to \mathrm{K}_*(B)\to(G,g_0)\to (\mathrm{K}_*(A),[1_A])\to 0.
$$
Then there exists an extension with trivial index maps of $A$ by $B$:
$$
0\to B \xrightarrow{\iota} E \xrightarrow{\pi}  A\to 0,
$$
such that the following diagram commutes:
$$
\xymatrixcolsep{2pc}
\xymatrix{
{\,\,0\,\,} \ar[r]^-{}
& {\,\,\mathrm{K}_*(B)\,\,} \ar[d]_-{{\rm id}} \ar[r]^-{}
& {\,\,(G,g_0)\,\,} \ar[d]_-{\alpha} \ar[r]^-{}
& {\,\,(\mathrm{K}_*(A), [1_A])\,\,} \ar[d]_-{{\rm id}} \ar[r]^-{}
& {\,\,0\,\,} \\
{\,\,0\,\,} \ar[r]^-{}
& {\,\,\mathrm{K}_*(B)\,\,} \ar[r]^-{\mathrm{K}_*(\iota)}
& {\,\,(\mathrm{K}_*(E),[1_E]) \,\,} \ar[r]^-{\mathrm{K}_*(\pi)}
& {\,\,(\mathrm{K}_*(A), [1_A]) \,\,} \ar[r]_-{}
& {\,\,0\,\,}.}
$$
Here, such an $E$ is unique up to isomorphism by Proposition \ref{strong tui cong}.

\section{Counterexample}
In this section, we construct two unital, separable, nuclear ${\rm C}^*$-algebras of stable rank one and  real rank zero with the same ordered scaled  total K-theory, and we show that they are not isomorphic.
\begin{notion}\label{b1b2}\rm
Recall the ${\rm C}^*$-algebras $C,D$ (with $p=2$) constructed in \cite[Theorem 3.3]{DL3}.

For any ${\rm C}^*$-algebra $A$ with unit $1_A$, denote $\delta_i\,:\,\widetilde{\mathbb{I}}_2\rightarrow A$, for $i=0,1$, the map $\delta_i(f)=f(i)\cdot 1_A$. If $0<t<1$, we use $\delta_{t}\,:\,\widetilde{\mathbb{I}}_2\rightarrow M_2(A)$ for the map $\delta_t(f)=1_A\otimes f(t)$. Let $\{t_n\}$ be a sequence dense in $(0,1)$, define $\psi_n={\rm id} \oplus\delta_{t_n}: \widetilde{\mathbb{I}}_2\rightarrow M_3(\widetilde{\mathbb{I}}_2)$, set
$$
\varphi_n=\psi_n\otimes {\rm id}: M_{3^{n+1}}(\widetilde{\mathbb{I}}_2)\rightarrow M_{3^{n+2}}(\widetilde{\mathbb{I}}_2).
$$

For any $n\in \mathbb{N}$, let
$$
A_n=M_{3^{n+1}}(\widetilde{\mathbb{I}}_2)\oplus M_{3^n}\oplus\cdots\oplus M_3\oplus \mathbb{C}\oplus M_3\oplus \cdots \oplus M_{3^n}.
$$
Define the connecting maps $\phi_{n,n+1}, \phi'_{n,n+1}: A_n\rightarrow A_{n+1}$ as follows
$$
\phi_{n,n+1}(f,x_{-n},\cdots,x_n)=(\varphi_n(f),f(0),x_{-n},\cdots,x_n,f(0)),
$$
$$
\phi'_{n,n+1}(f,x_{-n},\cdots,x_n)=(\varphi_n(f),f(0),x_{-n},\cdots,x_n,f(1)),
$$
where $f\in M_{3^{n+1}}(\widetilde{\mathbb{I}}_2)$, $x_0\in \mathbb{C}$ and $x_{-m},x_m \in M_{3^{m}}$ for $m=1,\cdots,n$.

Write $C=\lim\limits_{\longrightarrow}(A_n,\phi_{n,n+1})$, $D=\lim\limits_{\longrightarrow}(A_n,\phi'_{n,n+1})$.
 Then $C,D$ have real rank zero and stable rank one.

In this whole section, we use the convention  $$B_1:=C\otimes\mathcal{K},\quad B_2:=D\otimes\mathcal{K}. $$
It is proved in
\cite[Theorem 3.3]{DL3}  that $B_1$, $B_2$ are not isomorphic and
$$
(\mathrm{K}_*(B_1),\mathrm{K}_*^+(B_1))
\cong(\mathrm{K}_*(B_2),\mathrm{K}_*^+(B_2)).
$$
By \cite[Theorem 9.1]{DG}, we have
$$
(\underline{\mathrm{K}}(B_1),\underline{\mathrm{K}}(B_1)_+)\ncong
(\underline{\mathrm{K}}(B_2),\underline{\mathrm{K}}(B_2)_+),
$$
where the orders we take above are both Dadarlat-Gong orders. Note that by
Theorem \ref{kk-equivalence}, we can still have $\underline{\mathrm{K}}(B_1)\cong\underline{\mathrm{K}}(B_2)$, while the orders are not preserved.
\end{notion}

\begin{notion}\rm
Denote $\mathbf{Z}$ the subgroup of $\mathbb{Z}[\frac{1}{3}]\oplus \prod_{\mathbb{Z}}\mathbb{Z}$, which consists of all the elements of the form $(a, \prod_{\mathbb{Z}} a_m)$ with $a_m=3^{|m|} \cdot a$ for all large enough $|m|$.

For $B_1$, $B_2$, by \cite[Theorem 3.3]{DL3}, we have
$$
\mathrm{K}_0(B_i)= \mathbf{Z},\quad  \mathrm{K}_1(B_i)=\mathbb{Z}_2,
$$
and
$$
 \mathrm{K}_0^+(B_i)=\mathbf{Z}\cap ( \mathbb{R}_+ \oplus \prod_{\mathbb{Z}}\mathbb{R}_+),
$$
where $i=1,2$ and $\mathbb{R}_+=[0,+\infty)$.

As for each $i=1,2,$ $\mathrm{K}_0(B_i)$ is a torsion free group, and it is immediate that $\mathrm{K}_0(B_i)$
is weakly unperforated (see Definition \ref{0weakly un}).
\end{notion}


\begin{notion}\rm  ({\bf Counterexample})\label{counter ex}
Denote $\mathbf{Q}$ to be the subgroup of $\mathbb{Q}\oplus\prod_\mathbb{Z}\mathbb{Q}$ consisting of all the elements with the form $(a, \prod_{\mathbb{Z}} a_m)$ and $ a_m= 3^{|m|}\cdot a$  for
all large enough $|m|$.

Consider the following exact sequence
$$0\to\mathbf{Z}\xrightarrow{\zeta}\mathbf{Q}\xrightarrow{\eta} \mathbf{Q}_\chi^3\to 0,$$
where $\zeta(a, \prod_\mathbb{Z}a_m)=(a, \prod_\mathbb{Z}a_m)$
is the natural embedding map,
$\mathbf{Q}_\chi^3:=\mathbf{Q}/\mathbf{Z}$ is the quotient group and $\eta$ is the quotient map.

(Denote
$$\mathbb{Q}_\chi^3=\{[\frac{q}{p}]\in \mathbb{Q}/\mathbb{Z}:\,p,q\in \mathbb{Z},3\nmid p\}.$$
One can check that $\mathbf{Q}_\chi^3$ is isomorphic to the subgroup of $\mathbb{Q}_\chi^3\oplus\prod_\mathbb{Z}\mathbb{Q}/\mathbb{Z}$ consisting of all the elements of the form $(b, \prod_\mathbb{Z}b_m) $ and $b_m= 3^{|m|}\cdot b$  for all large enough $|m|$.)

We only use the fact that $\mathbf{Q}_\chi^3$  is a countable torsion group.
By \cite[Theorem 4.20]{EG}, let $A_\chi$ be a unital simple AH algebra of stable rank one and real rank zero, with no dimension growth satisfying
 $$\mathrm{K}_0(A_\chi)=\mathbb{Q}\oplus \mathbf{Q}_\chi^3,\quad \mathrm{K}_1(A_\chi)=0,\quad [1_{A_\chi}]=(1,0),$$
 $$\mathrm{K}_0^+(A_\chi)=\{(x,y)\in \mathbb{Q}\oplus \mathbf{Q}_\chi^3|\,x>0\}\cup\{(0,0)\in \mathbb{Q}\oplus \mathbf{Q}_\chi^3\}.$$
Let $\varepsilon_i\in {\rm Ext}_{[1]}(\mathrm{K}_*(A_\chi), \mathrm{K}_*(B_i))$ be the equivalent class of
$$
0\to \mathrm{K}_0(B_i)\xrightarrow{\zeta_i} (\mathbb{Q}\oplus \mathbf{Q},(1,0))\xrightarrow{\eta_i} (\mathrm{K}_0(A_\chi),(1,0))\to 0,
$$
$$
0\to \mathrm{K}_1(B_i)\xrightarrow{\nu_i}\mathbb{Z}_2\xrightarrow{0} \mathrm{K}_1(A_\chi)\to 0,
$$
where $\zeta_i(a, \prod_\mathbb{Z}a_m)=(0, (a, \prod_\mathbb{Z}a_m))$, $\nu_i= {\rm id}_{\mathbb{Z}_2}$ and
$$\eta_i(x, (b, \prod_\mathbb{Z}b_m))=(x, \eta(b, \prod_\mathbb{Z}b_m)).$$

By Theorem \ref{strong wei}, there are two unital essential extensions of ${\rm C}^*$-algebras with trivial index maps (see Definition \ref{trivial index def})
$$
0\to B_i\xrightarrow{\iota_i} E_i\xrightarrow{\pi_i}A_\chi\to 0, \quad i=1,2
$$
realising $\varepsilon_i$ and satisfying the following commutative diagram
$$
\xymatrixcolsep{1pc}
\xymatrix{
{\,\,0\,\,} \ar[r]^-{}
& {\,\,\mathrm{K}_*(B_i)\,\,} \ar[d]_-{{\rm id}} \ar[r]^-{(\zeta_i,\nu_i)}
& {\,\,((\mathbb{Q}\oplus \mathbf{Q},\mathbb{Z}_2),(1,0))\,\,} \ar[d]_-{\rho_i} \ar[r]^-{(\eta_i,0)}
& {\,\,(\mathrm{K}_*(A_\chi), [1_{A_\chi}])\,\,} \ar[d]_-{{\rm id}} \ar[r]^-{}
& {\,\,0\,\,} \\
{\,\,0\,\,} \ar[r]^-{}
& {\,\,\mathrm{K}_*(B_i)\,\,} \ar[r]^-{\mathrm{K}_*(\iota_i)}
& {\,\,(\mathrm{K}_*(E_i),[1_{E_i}]) \,\,} \ar[r]^-{\mathrm{K}_*(\pi_i)}
& {\,\,(\mathrm{K}_*(A_\chi), [1_{A_\chi}]) \,\,} \ar[r]_-{}
& {\,\,0\,\,}.}
$$
We will identify
$\mathrm{K}_*(E_i)=(\mathbb{Q}\oplus \mathbf{Q},\mathbb{Z}_2)$ through the isomorphism $\rho_i$.
Note that by Proposition \ref{strong tui cong}, for each $i=1,2$, $E_i$ is unique up to isomorphism
and by Proposition \ref{lin inj}, both $E_1,E_2$ are of stable rank one and real rank zero.

{\bf In this whole section, $E_1$, $E_2$ and $A_\chi$ will always be taken as the above ${\rm C}^*$-algebras.}

\end{notion}
\begin{remark}
As $B_1,B_2, A_\chi$ are separable nuclear ${\rm C}^*$-algebras of stable rank one and real rank zero in the bootstrap class $\mathcal{N}$, we have $E_1,E_2$ are both unital, separable, nuclear ${\rm C}^*$-algebras of stable rank one and real rank zero  in  $\mathcal{N}$ (for each $E_i$, separability is trivial; see nuclearity in \cite[IV 3.1.3]{Bop}
).
For each $i=1,2$, the following exact sequence
$$
0\to \mathrm{K}_0(B_i)\xrightarrow{{\rm K}_0(\iota_i)} \mathrm{K}_0(E_i)\xrightarrow{{\rm K}_0(\pi_i)}\mathrm{K}_0(A_\chi)\to 0,
$$
which is
$$
0\to \mathbf{Z} \xrightarrow{0\oplus\zeta} \mathbb{Q}\oplus\mathbf{Q}\xrightarrow{{\rm id}\oplus\eta}\mathbb{Q}\oplus \mathbf{Q}^3_\chi\to 0,
$$
is not a pure  extension of abelian groups (see Definition \ref{group pure}). So both $E_1$ and $E_2$ are not K-pure ${\rm C}^*$-algebras
(see Definition \ref{def kpure}). By \cite[Proposition 4.4]{DE}, we have both of them are not $A\mathcal{HD}$ algebras. We inform the readers in advance
that by Lemma \ref{k-pure lemma} (iii) and (iv), for each $i=1,2,$ in the following sequence
$$\underline{{\rm K}}(B_i ) \xrightarrow{\underline{{\rm K}}(\iota_i)} \underline{{\rm K}}(E_i) \xrightarrow{\underline{{\rm K}}(\pi_i)} \underline{{\rm K}}(A_\chi), $$
$\underline{{\rm K}}(\iota_i)$ is not injective, and at the same time $\underline{{\rm K}}(\pi_i)$ is not surjective.
\end{remark}
\begin{proposition}\label{Ei ktotal}
For each $i=1,2$, we have
$$
\mathrm{K}_0(E_i)=\mathbb{Q}\oplus \mathbf{Q}\quad{\rm and}\quad
\mathrm{K}_1(E_i)=\mathbb{Z}_2;$$ and for any $n\geq1$,
$$
\mathrm{K}_0(E_i;\mathbb{Z}_{2n})=\mathbb{Z}_{2},\quad
\mathrm{K}_1(E_i;\mathbb{Z}_{2n})=\mathbb{Z}_{2}
$$
and
$$
\mathrm{K}_0(E_i;\mathbb{Z}_{2n-1})=0,\quad
\mathrm{K}_1(E_i;\mathbb{Z}_{2n-1})=0.
$$
\end{proposition}
\begin{proof}
Recall that in \ref{counter ex}, for each $i=1,2$, we have
$$
\mathrm{K}_0(E_i)=\mathbb{Q}\oplus \mathbf{Q}\quad{\rm and}\quad
\mathrm{K}_1(E_i)=\mathbb{Z}_2.$$
For each $k\geq1$,
via  the cofibre sequence, we have the following exact sequence
$$
\xymatrixcolsep{3pc}
\xymatrix{
{\mathrm{K}_0(E_i)}  \ar[r]^-{\times k}
& {\mathrm{K}_0(E_i)}  \ar[r]^-{}
& {\mathrm{K}_0(E_i; \mathbb{Z}_k)} \ar[d]_-{}
 \\
{\mathrm{K}_1(E_i; \mathbb{Z}_k)} \ar[u]_-{}
& {\mathrm{K}_1(E_i)} \ar[l]_-{}
& {\mathrm{K}_1(E_i)} \ar[l]_-{\times k}.
}
$$
Note that the map
$\times k:\,\mathbb{Q}\oplus \mathbf{Q}\to \mathbb{Q}\oplus \mathbf{Q}$ in the first row is an isomorphism.
If $k$ is even, the map $\times k:\,\mathbb{Z}_2\to \mathbb{Z}_2$ in the second row is the zero map; if $k$ is odd,  the map $\times k:\,\mathbb{Z}_2\to \mathbb{Z}_2$ is an isomorphism.

By exactness of the diagram, it is immediate that for any $n\geq1$,
$$
\mathrm{K}_0(E_i;\mathbb{Z}_{2n})=\mathbb{Z}_{2},\quad
\mathrm{K}_1(E_i;\mathbb{Z}_{2n})=\mathbb{Z}_{2}
$$
and
$$
\mathrm{K}_0(E_i;\mathbb{Z}_{2n-1})=0,\quad
\mathrm{K}_1(E_i;\mathbb{Z}_{2n-1})=0.
$$
\end{proof}

\begin{notion}\rm\label{def lambda}
Now we define $$
\lambda:\,\underline{\mathrm{K}}(E_1)\to\underline{\mathrm{K}}(E_2)
$$
 to be the identity map from the above proposition.

 Note that by  UCT (Theorem \ref{UCT})
and  UMCT (Theorem \ref{UMCT}), every isomorphism between the $\mathrm{K}_*$-groups
can be lifted to a
$\Lambda$-isomorphism between the corresponding $\underline{\mathrm{K}}$-groups. Also note that $\lambda$ is the unique graded isomorphism from $\underline{\mathrm{K}}(E_1)$ to $\underline{\mathrm{K}}(E_2)$, whose restriction on $\mathrm{K}_*$-groups is ${\rm id}\in {\rm Hom}(\mathrm{K}_*(E_1),\mathrm{K}_*(E_2))$. Hence, $\lambda$ is  automatically a $\Lambda$-isomorphism.
\end{notion}


\begin{lemma}{\rm (}\cite[Corollary 2.7]{AL2}{\rm )}\label{sheying dadada}
Let $A,B$ be nuclear separable ${\rm C}^*$-algebras of stable rank one and real rank zero. Suppose that $A$ is unital simple,  $B$ is stable and $({\rm K}_0(B),{\rm K}_0^+(B))$  is weakly unperforated (see Definition \ref{0weakly un}). Let $e$ be a unital essential extension with trivial index maps (see Definition \ref{trivial index def})
$$
0\to B\xrightarrow{\iota} E\xrightarrow{\pi} A\to 0.
$$
Then for any nonzero projection $p\in E\backslash B$ and any projection $q\in B$, we have $[\iota(q)]\leq [p]$ in ${\rm K}_0(E)$.
\end{lemma}
\begin{theorem}\label{yibande k0+}
Let $A,B$ be nuclear, separable ${\rm C}^*$-algebras of stable rank one and real rank zero. Suppose that $A$ is unital simple,  $B$ is stable and $({\rm K}_0(B),{\rm K}_0^+(B))$  is weakly unperforated. Let $e$ be a unital essential extension with trivial index maps
$$
0\to B\xrightarrow{\iota} E\xrightarrow{\pi} A\to 0.
$$
Then
$$\mathrm{K}_0^+(E)=((\mathrm{K}_0(\pi))^{-1} (\mathrm{K}_0^+(A)\backslash\{0\}))
\,\cup\,  (\mathrm{K}_0(\iota)(\mathrm{K}_0^+(B))).$$

\end{theorem}
\begin{proof}
Note that
$$
0\to B\xrightarrow{\iota} E\xrightarrow{\pi} A\to 0
$$
is a unital essential extension of ${\rm C}^*$-algebras with trivial index maps.
By Proposition \ref{lin inj}, we have $E$ has stable rank one and real rank zero. In particular, we have the following is an exact sequence of abelian groups
$$
0\to \mathrm{K}_0(B)\xrightarrow{\mathrm{K}_0(\iota)} \mathrm{K}_0(E)\xrightarrow{\mathrm{K}_0(\pi)} \mathrm{K}_0(A)\to 0.
$$
Since $\iota,\pi$ are homomorphisms, we have
$$
\mathrm{K}_0(\iota)(\mathrm{K}_0^+(B))\subset \mathrm{K}_0^+(E)
$$
and also by exactness,
$$
  \mathrm{K}_0^+(E)\subset
  ((\mathrm{K}_0(\pi))^{-1}(\mathrm{K}_0^+(A)\backslash\{0\}))\,\cup \, (\mathrm{K}_0(\iota)(\mathrm{K}_0(B))).
$$

For any
$
x\in \mathrm{K}_0(E)
$
with $\mathrm{K}_0(\pi)(x)\in \mathrm{K}_0^+(A)\backslash\{0\}$,
$\mathrm{K}_0(\pi)(x)$ can be lifted to a projection in $M_m(A)$ for some integer $m$.
Considering the induced extension as follows:
$$
0\to M_m(B)\xrightarrow{\iota} M_m(E)\xrightarrow{\pi}M_m(A)\to 0,
$$
in which we still denote  the induced maps by
$\iota,\pi$, respectively.
By Proposition \ref{lin inj} (i)(c),
there exists a projection $g\in M_m(E)$ such that
$$\mathrm{K}_0(\pi)([g])=\mathrm{K}_0(\pi)(x)\in \mathrm{K}_0^+(A).$$
By exactness, there exist projections
$p_0,q_0\in M_m(B)\cong B$ such that
$$
[g]=x+[\iota(q_0)]-[\iota(p_0)]\in \mathrm{K}_0^+(E).
$$
Then
by Lemma \ref{sheying dadada},
for any projection $f\in M_m(B)\cong B$, we have
$$[\iota(f)]\leq [g]=x+[\iota(q_0)]-[\iota(p_0)]\in \mathrm{K}_0^+(E).$$
Particularly, we choose $f\in M_m(B)\cong B$ with $[f]=[q_0].$ Since $E$ has cancellation of projections,
we obtain that
$$x-[\iota(p_0)]\in \mathrm{K}_0^+(E).$$
Hence,
$x\in \mathrm{K}_0^+(E).$
That is,
$$(\mathrm{K}_0(\pi))^{-1}(\mathrm{K}_0^+(A)\backslash\{0\})\subset \mathrm{K}_0^+(E).
$$

Note that $$(\mathrm{K}_0(\pi))^{-1}(0)=\mathrm{K}_0(\iota)(\mathrm{K}_0(B)).$$
Take any $y\in \mathrm{K}_0(B)$ with $\mathrm{K}_0(\iota)(y)\in \mathrm{K}_0^+(E)$. Since
$E$ has stable rank one, $y$ can be realised as a projection $h$ in $E\otimes \mathcal{K}$.
Since
$$\mathrm{K}_0(\pi)([1_E]-[h])=[1_A]\in \mathrm{K}_0^+(A)\backslash\{ 0\},$$
as we have just shown above, we obtain
$$
[1_{E}]-[h]=[1_{E}]-y\in \mathrm{K}_0^+(E).
$$
Then there exists a unitary $u\in M_n(E)$ for some integer $n$ such that $uhu^*\in E$. 
   Since $A$ has stable rank one, then $[\pi(uhu^*)]=0$ implies $\pi(uhu^*)=0$, that is, by exactness, $\iota^{-1}(uhu^*)\in B$ realises $y$. Now we have
$$
((\mathrm{K}_0(\pi))^{-1}(0))\cap \mathrm{K}_0^+(E)= \mathrm{K}_0(\iota)(\mathrm{K}_0^+(B)).
$$
This concludes the proof.

\end{proof}
The following proposition is a direct application of Theorem \ref{yibande k0+}.
\begin{proposition}\label{positive cone} For each $i=1,2$,
$$\mathrm{K}_0^+(E_i)=\{(\frac{l}{k},y)|k,l\in\mathbb{N}\backslash\{0\}, y\in \mathbf{Q}\}\cup \{(0,y)|y\in \mathrm{K}_0^+(B_i)\}.$$
Hence,
$$\mathrm{K}_0^+(E_1)\cong \mathrm{K}_0^+(E_2).$$
Moreover,  for each $i=1,2,$ $(\mathrm{K}_0(E_i),\mathrm{K}_0^+(E_i))$ is torsion free and is not weakly unperforated (see Definition \ref{0weakly un}).
\end{proposition}


Now let us check that $\lambda:\,\underline{{\rm K}}(E_1)\to \underline{{\rm K}}(E_2)$ is an ordered scaled $\Lambda$-isomorphism as one of our main results:
\begin{theorem}\label{zhuyao fanli}
We have $$
(\underline{\mathrm{K}}(E_1),\underline{\mathrm{K}}(E_1)_+,[1_{E_1}])\cong
(\underline{\mathrm{K}}(E_2),\underline{\mathrm{K}}(E_2)_+,[1_{E_2}]),
$$
while $E_1\ncong E_2$.
\end{theorem}
\begin{proof}
By \ref{def lambda}, since $\lambda$ is already a $\Lambda$-linear isomorphism, we only  need to check $\lambda$ is also an order-preserving isomorphism.

Given any
$$\textstyle
(x,\mathfrak{u},\bigoplus\limits_{n=1}^{\infty}  ( \mathfrak{s}_{n,0}, \mathfrak{s}_{n,1}))\in \underline{\mathrm{K}}(E_1)_+,
$$
we are going to show
$$\textstyle
\lambda(x,\mathfrak{u},\bigoplus\limits_{n=1}^{\infty} ( \mathfrak{s}_{n,0}, \mathfrak{s}_{n,1}))\in \underline{\mathrm{K}}(E_2)_+.
$$
By Dadarlat-Gong order \ref{dg order}, we have
$x\in \mathrm{K}_0^+(E_1)$, $\mathfrak{u}\in \mathrm{K}_1(I_{x}|E_1)$, $\mathfrak{s}_{n,0}\in \mathrm{K}_0(I_{x}|E_1;\mathbb{Z}_n),$
$\mathfrak{s}_{n,1}\in \mathrm{K}_1(I_{x}|E_1;\mathbb{Z}_n),$ where $I_{x}$ is the ideal of $E_1\otimes \mathcal{K}$ generated by $x$.
Note that at most finitely many elements of $x,\mathfrak{u},\mathfrak{s}_{n,0},\mathfrak{s}_{n,1}$ ($n\geq1$) are nonzero.

By Proposition \ref{positive cone}, we have
$$
x\in \mathrm{K}_0^+(E_1)=\{(\frac{l}{k},y)\,|\,k,l\in\mathbb{N}\backslash\{0\}, y\in \mathbf{Q}\}\cup \{(0,y)\,|\,y\in \mathrm{K}_0^+(B_1)\}.
$$
We consider the following two cases:

Case 1: If $x\in \{(\frac{l}{k},y)\,|\,k,l\in\mathbb{N}\backslash\{0\}, y\in \mathbf{Q}\}$, then
$x\geq (0,y')\in \mathrm{K}_0^+(E_1)$ for any $y'\in \mathrm{K}_0^+(B_1)$. Thus,
$I_x=E_1\otimes \mathcal{K}$ ($A_\chi$ is simple).

As $\lambda_0^0:{\rm K}_0(E_1)\to{\rm K}_0(E_2)$ is the identity map, we have $I_{\lambda_0^0(x)}=E_2\otimes \mathcal{K}$.
$\lambda$ is an isomorphism, of cause we have
$$\textstyle
\lambda(x,\mathfrak{u},\bigoplus\limits_{n=1}^{\infty} ( \mathfrak{s}_{n,0}, \mathfrak{s}_{n,1}))\in \underline{\mathrm{K}}(E_2)=\underline{\mathrm{K}}(I_{\lambda_0^0(x)}\mid E_2).
$$

Case 2: If $x\in \{(0,y)\,|\,y\in \mathrm{K}_0^+(B_1)\}$, $y$ has the form $(a,\prod_{\mathbb{Z}} a_m)\in  \mathrm{K}_0^+(B_1)$.

Case 2.1: If $a=0$, then only finitely many $a_m$ are nonzero.
The ideal $I_{x}$  generated by
$x$ in $E_1\otimes \mathcal{K}$,
which is exactly the ideal generated by
$(a,\prod_{\mathbb{Z}} a_m)$
in $B_1$ ($B_1$ is stable), is isomorphic to a finite direct sum of $\mathcal{K}$ (see the construction of $B_1$ in \ref{b1b2}).
Then the natural embedding map $\iota_{I_x}:\,I_x\to E_1\otimes \mathcal{K}$
will induce the following commutative diagram with exact rows
$$
\xymatrixcolsep{2pc}
\xymatrix{
{\mathrm{K}_0(I_x)} \ar[d]^-{{\rm K}_0(\iota_{I_x})} \ar[r]^-{\times n}
& {\mathrm{K}_0(I_x)} \ar[d]^-{{\rm K}_0(\iota_{I_x})} \ar[r]^-{}
& {\mathrm{K}_0(I_x;\mathbb{Z}_n)} \ar[d]^-{{\rm K}_0(\iota_{I_x};\mathbb{Z}_n)} \ar[r]^-{}
& 
 {\mathrm{K}_1(I_x)} \ar[d]^-{{\rm K}_1(\iota_{I_x})} \ar[r]^-{\times n}
& {\mathrm{K}_1(I_x)} \ar[d]^-{{\rm K}_1(\iota_{I_x})}
 \\
{\mathrm{K}_0(E_1)} \ar[r]^-{\times n}
& {\mathrm{K}_0(E_1)} \ar[r]_-{}
& {\mathrm{K}_0(E_1; \mathbb{Z}_n)} \ar[r]^-{}
& 
{\mathrm{K}_{1}(E_1)} \ar[r]^-{\times n}
& {\mathrm{K}_{1}(E_1)},
}
$$
which is
$$
\xymatrixcolsep{2pc}
\xymatrix{
{\mathrm{K}_0(I_x)} \ar[d]^-{{\rm K}_0(\iota_{I_x})} \ar[r]^-{\times n}
& {\mathrm{K}_0(I_x)} \ar[d]^-{{\rm K}_0(\iota_{I_x})} \ar[r]^-{}
& {\mathrm{K}_0(I_x;\mathbb{Z}_n)} \ar[d]^-{{\rm K}_0(\iota_{I_x};\mathbb{Z}_n)} \ar[r]^-{}
& 
 {0} \ar[d]^-{{\rm K}_1(\iota_{I_x})} \ar[r]^-{\times n}
& {0} \ar[d]^-{{\rm K}_1(\iota_{I_x})}
 \\
{\mathbb{Q}\oplus \mathbf{Q}} \ar[r]^-{\times n}
& {\mathbb{Q}\oplus \mathbf{Q}} \ar[r]_-{}
& {\mathrm{K}_0(E_1; \mathbb{Z}_n)} \ar[r]^-{\beta_n^{0,E_1}}
& 
{\mathbb{Z}_2} \ar[r]^-{\times n}
& {\mathbb{Z}_2}.
}
$$
In the above diagram: when $n$ is even ($n\geq 2$),
the map $\times n:\,\mathbb{Q}\oplus \mathbf{Q}\to\mathbb{Q}\oplus \mathbf{Q} $ is an isomorphism and
the map $\times n:\,\mathbb{Z}_2\to \mathbb{Z}_2$ is a zero map. By the exactness of the second row,
we have $\beta_n^{0,E_1}:\,\mathrm{K}_0(E_1; \mathbb{Z}_n)\to \mathbb{Z}_2$ is an isomorphism.
Hence, by commutativity, we have ${\rm K}_0(\iota_{I_x};\mathbb{Z}_n)$ is a zero map;
when $n$ is odd ($n\geq 1$),
by Proposition \ref{Ei ktotal},
we obtain $\mathrm{K}_0(E_1; \mathbb{Z}_n)=0$, so ${\rm K}_0(\iota_{I_x};\mathbb{Z}_n)$ is also a zero map. (${\rm K}_1(\iota_{I_x};\mathbb{Z}_n)$ is always a zero map, as $\mathrm{K}_1(I_x; \mathbb{Z}_n)$ is always a zero group.)

This means that we can identify  $\underline{\mathrm{K}}(I_x\,|\, E_1)$ with $\mathrm{K}_0(I_x)$, even though $\underline{\mathrm{K}}(I_x)$ contains more information than $\mathrm{K}_0(I_x)$ for any nonzero $x=(0,y)$.
By Dadarlat-Gong order \ref{dg order}, the positivity of $(x,\mathfrak{u},\bigoplus_{n=1}^{\infty}  ( \mathfrak{s}_{n,0}, \mathfrak{s}_{n,1}))$ implies that all $\mathfrak{u},\mathfrak{s}_{n,0},\mathfrak{s}_{n,1}$ ($n\geq1$) are zero, and we have
$$\textstyle
\lambda(x,\mathfrak{u},\bigoplus\limits_{n=1}^{\infty} ( \mathfrak{s}_{n,0}, \mathfrak{s}_{n,1}))
=
(\lambda_0^0(x),0,\bigoplus\limits_{n=1}^{\infty} (0,0))
\in\underline{\mathrm{K}}(I_{\lambda_0^0(x)}\mid E_2).
$$

Case 2.2: If $a>0$, we have only finitely many $a_m$ are zero.
Denote the ideal $I_{x}$  generated by
$(a,\prod_{\mathbb{Z}}a_m)$ in $E_1\otimes \mathcal{K}$, which is also an ideal of $B_1$.
From the construction of $B_1$ in \ref{b1b2}, there exists
a ${\rm C}^*$-algebra $R_1$, which is a finite direct sum of $\mathcal{K}$ corresponding to the zero entry of $\prod_{\mathbb{Z}} a_m$, satisfying
$$
B_1=R_1\oplus I_{x}.
$$
Similarly, consider the diagram
$$
\xymatrixcolsep{1pc}
\xymatrix{
{\mathrm{K}_0(B_1)} \ar[d]^-{{\rm K}_0(\iota_1)} \ar[r]^-{\times n}
& {\mathrm{K}_0(B_1)} \ar[d]^-{{\rm K}_0(\iota_1)} \ar[r]^-{}
& {\mathrm{K}_0(B_1;\mathbb{Z}_n)} \ar[d]^-{{\rm K}_0(\iota_1;\mathbb{Z}_n)} \ar[r]^-{}
& 
 {\mathrm{K}_1(B_1)} \ar[d]^-{{\rm K}_1(\iota_1)} \ar[r]^-{\times n}
& {\mathrm{K}_1(B_1)} \ar[d]^-{{\rm K}_1(\iota_1)} \ar[r]^-{}
& {\mathrm{K}_1(B_1;\mathbb{Z}_n)} \ar[d]^-{{\rm K}_1(\iota_1;\mathbb{Z}_n)}
 \\
{\mathrm{K}_0(E_1)} \ar[r]^-{\times n}
& {\mathrm{K}_0(E_1)} \ar[r]_-{}
& {\mathrm{K}_0(E_1; \mathbb{Z}_n)} \ar[r]^-{}
& 
{\mathrm{K}_{1}(E_1)} \ar[r]^-{\times n}
& {\mathrm{K}_{1}(E_1)}\ar[r]^-{}
& {\mathrm{K}_1(B_1;\mathbb{Z}_n)} ,
}
$$
which is
$$
\xymatrixcolsep{2pc}
\xymatrix{
{\mathbf{Z}} \ar[d]^-{{\rm K}_0(\iota_1)} \ar[r]^-{\times n}
& {\mathbf{Z}} \ar[d]^-{{\rm K}_0(\iota_1)} \ar[r]^-{}
& {\mathrm{K}_0(B_1; \mathbb{Z}_n)} \ar[d]^-{{\rm K}_0(\iota_1;\mathbb{Z}_n)} \ar[r]^-{\beta_n^{0,B_1}}
& 
 {\mathbb{Z}_2} \ar[d]^-{{\rm id}} \ar[r]^-{\times n}
& {\mathbb{Z}_2} \ar[d]^-{{\rm id}}\ar[r]^-{\rho_n^{1,B_1}}
& {\mathrm{K}_1(B_1; \mathbb{Z}_n)} \ar[d]^-{{\rm K}_1(\iota_1;\mathbb{Z}_n)}
 \\
{\mathbb{Q}\oplus \mathbf{Q}} \ar[r]^-{\times n}
& {\mathbb{Q}\oplus \mathbf{Q}} \ar[r]_-{}
& {\mathrm{K}_0(E_1; \mathbb{Z}_n)} \ar[r]^-{\beta_n^{0,E_1}}
& 
{\mathbb{Z}_2} \ar[r]^-{\times n}
& {\mathbb{Z}_2} \ar[r]^-{\rho_n^{1,E_1}}
& {\mathrm{K}_1(E_1; \mathbb{Z}_n)}.
}
$$
(Recall the construction in \ref{counter ex}, we take ${\rm K}_1(\iota_1)={\rm id}_{\mathbb{Z}_2}$.)
In the above diagram: when $n$ is odd,
by Proposition \ref{Ei ktotal}, we have $\mathrm{K}_0(E_1; \mathbb{Z}_n)=0$ and $\mathrm{K}_1(E_1; \mathbb{Z}_n)=0$.
So both ${\rm K}_0(\iota_1;\mathbb{Z}_n)$ and ${\rm K}_1(\iota_1;\mathbb{Z}_n)$ are automatically surjective.
When $n$ is even ($n\geq 2$), the maps
$\times n:\,\mathbb{Q}\oplus \mathbf{Q}\to\mathbb{Q}\oplus \mathbf{Q} $
 and
$\times n:\,\mathbb{Z}_2\to \mathbb{Z}_2$ are isomorphisms and
$\times n:\,\mathbf{Z}\to \mathbf{Z} $ is injective. By exactness of both the rows,
we have
$\beta_n^{0,E_1},\rho_n^{1,E_1},\rho_n^{1,B_1}$ are all  isomorphisms and $\beta_n^{0,B_1}$ is surjective.
By commutativity, we have both ${\rm K}_0(\iota_1;\mathbb{Z}_n)$ and ${\rm K}_1(\iota_1;\mathbb{Z}_n)$ are surjective.

Now for any $n\geq 1$,
as we have discussed in Case 2.1,
both the restriction of ${\rm K}_0(\iota_1;\mathbb{Z}_n)$ on ${\rm K}_0(R_1;\mathbb{Z}_n)$ and the restriction of ${\rm K}_1(\iota_1;\mathbb{Z}_n)$
on  ${\rm K}_1(R_1;\mathbb{Z}_n)$  are zero maps, which means
 the restriction of ${\rm K}_0(\iota_1;\mathbb{Z}_n)$ on ${\rm K}_0(I_x;\mathbb{Z}_n)$ and the restriction of ${\rm K}_1(\iota_1;\mathbb{Z}_n)$ on ${\rm K}_1(I_x;\mathbb{Z}_n)$   are surjective.
Also note that the restriction of ${\rm K}_1(\iota_1)$ on ${\rm K}_1(I_x)$ is also surjective.
That is,
$$\underline{\mathrm{K}}(B_1\,|\, E_1)=
\underline{\mathrm{K}}(R_1\oplus I_{x}\,|\, E_1)=\underline{\rm K}(R_1\,|\,E_1)\oplus \underline{\mathrm{K}}(I_{x}\,|\, E_1).$$
Similar with the proof of  Case 2.1, we can identify $\underline{\rm K}(R_1\,|\,E_1)$ with ${\rm K}_0(R_1)$.
This implies that for any choice of $\mathfrak{v}\in \mathrm{K}_1(E_1),\mathfrak{t}_{n,0}\in \mathrm{K}_0(E_1;\mathbb{Z}_n),\mathfrak{t}_{n,1}\in \mathrm{K}_1(E_1;\mathbb{Z}_n)$, we will always have
$$\textstyle
(x,\mathfrak{v},\bigoplus\limits_{n=1}^{\infty} (\mathfrak{t}_{n,0}, \mathfrak{t}_{n,1}))\in
\underline{\mathrm{K}}(I_x\mid E_1)
$$
is a positive element with respect to the
Dadarlat-Gong order \ref{dg order}.

As $\lambda_0^0:{\rm K}_0(E_1)\to{\rm K}_0(E_2)$ is the identity map, with the same procedure  for $\lambda_0^0(x)$ and $E_2$, we obtain
$$\textstyle
\lambda(x,\mathfrak{u},\bigoplus\limits_{n=1}^{\infty} ( \mathfrak{s}_{n,0}, \mathfrak{s}_{n,1}))
=(\lambda_0^0(x),\lambda_0^1(\mathfrak{u}),\bigoplus\limits_{n=1}^{\infty} (\lambda_n^0(\mathfrak{s}_{n,0}), \lambda_n^1(\mathfrak{s}_{n,1})))
\in\underline{\mathrm{K}}(I_{\lambda_0^0(x)}\mid E_2)
$$
is a positive element in $\underline{\mathrm{K}}(E_2)$.

Combining Case 1, Case 2.1 and Case 2.2, by  \ref{dg order},
$\lambda$ preserves Dadarlat-Gong order. The converse direction is the same.

In general, we have $$
(\underline{\mathrm{K}}(E_1),\underline{\mathrm{K}}(E_1)_+,[1_{E_1}])\cong
(\underline{\mathrm{K}}(E_2),\underline{\mathrm{K}}(E_2)_+,[1_{E_2}]).
$$

More directly, we have $(x,\mathfrak{u},\bigoplus_{n=1}^{\infty}  ( \mathfrak{s}_{n,0}, \mathfrak{s}_{n,1}))$
is positive if and only if one of the following is satisfied:

(1) $x=(\frac{l}{k},y)$ and $\frac{l}{k}>0$;

(2) $x=(0,y)$, $y=(a,\prod_{\mathbb{Z}} a_m)\in {\rm K}_0^+(B_i)$ and $a>0$;

(3) $x=(0,y)$, $y=(0,\prod_{\mathbb{Z}} a_m)\in {\rm K}_0^+(B_i)$, $\mathfrak{u}=0,$ $\mathfrak{s}_{n,0}=0,$ $\mathfrak{s}_{n,1}=0$ for all  $n\geq1$.

As $E_1,E_2$ are separable ${\rm C}^*$-algebras of real rank zero and stable rank one, all the ideals of them are generated by projections.
By Proposition \ref{positive cone}, we point out that
$B_1$, $B_2$ are the unique maximal ideals of $E_1$, $E_2$, respectively,
which implies
$$E_1\ncong E_2.$$
Otherwise, suppose that there is an isomorphism from $E_1$ to $E_2$, it must take the unique maximal ideal $B_1$ isomorphic to  the unique maximal ideal $B_2$. However, $B_1\ncong B_2$, which forms a contradiction.

\end{proof}
\begin{remark}
Now we have shown that the total K-theory can't distinguish $E_1$ and $E_2$, then the Conjecture \ref{ell conj} is not true. Recall the construction in \ref{counter ex}, the extensions we choose are not pure extensions and $E_1,E_2$ are not K-pure,  the total K-theory doesn't fully reflect the
structure of ideals. Under the ``K-pure'' setting, this would not happen. It can be expected that Elliott Conjecture for real rank zero ${\rm C}^*$-algebras may hold  for K-pure algebras.  In the last section, we will use a new invariant called total Cuntz semigroup to distinguish $E_1$, $E_2$ and classify non K-pure extensions. This gives evidence supporting this new invariant might be a complete invariant for separable nuclear ${\rm C}^*$-algebras of stable rank one and real rank zero.
\end{remark}

\section{K-pureness and total K-theory}
In this section, some elementary properties are established for K-pure extensions and K-pure ${\rm C}^*$-algebras.

\begin{lemma}\label{k-pure lemma}
Given an $extension$ of ${\rm C}^*$-algebras
  $$
e:\quad 0 \to B \xrightarrow{\iota}  E \xrightarrow{\pi} A \to 0,$$
we have the following statements are equivalent:

(i) $e$ is a {\rm K}-pure extension;

(ii) For any $n \geq 0$, the sequences
$$0 \to {\rm K}_*(B ; \mathbb{Z}_n) \to {\rm K}_*(E; \mathbb{Z}_n) \to {\rm K}_*(A ; \mathbb{Z}_n) \to 0$$
are exact;

(iii) For any $n \geq 0$, $j=0,1,$ $\mathrm{K}_j(\iota; \mathbb{Z}_n)$ is an injective map, i.e.,
$\underline{\mathrm{K}}(\iota)$ is injective;

(iv) For any $n \geq 0$, $j=0,1,$ $\mathrm{K}_j(\pi; \mathbb{Z}_n)$ is a surjective map, i.e.,
$\underline{\mathrm{K}}(\pi)$ is surjective;

(v) The sequence  $$0 \to \underline{{\rm K}}(B ) \to \underline{{\rm K}}(E) \to \underline{{\rm K}}(A) \to 0$$
is exact.
\end{lemma}
\begin{proof}
For any $n\geq 0$, $j=0,1,$ from the following exact sequence
$$
\xymatrixcolsep{3pc}
\xymatrix{
{\mathrm{K}_j(B; \mathbb{Z}_n)}  \ar[r]^-{\mathrm{K}_j(\iota\,;\, \mathbb{Z}_n)}
& {\mathrm{K}_j(E; \mathbb{Z}_n)}  \ar[r]^-{}
& {\mathrm{K}_j(A; \mathbb{Z}_n)} \ar[d]_-{}
 \\
{\mathrm{K}_{1-j}(A; \mathbb{Z}_n)} \ar[u]_-{}
& {\mathrm{K}_{1-j}(E; \mathbb{Z}_n)} \ar[l]_-{\mathrm{K}_{1-j}(\pi\,;\, \mathbb{Z}_n)}
& {\mathrm{K}_{1-j}(B; \mathbb{Z}_n)} \ar[l]_-{},
}
$$
we have $\mathrm{K}_j(\iota; \mathbb{Z}_n)$ is injective
if and only if $\mathrm{K}_{1-j}(\pi; \mathbb{Z}_n)$ is surjective.

Thus, (ii) $\Leftrightarrow$ (iii) $\Leftrightarrow$ (iv) $\Leftrightarrow$ (v)
 is trivial. One can see \cite[Lemma 5.6]{AL} or \cite[Section 4]{DE} for (i) $\Rightarrow$ (ii). Now we prove (iii) $\Rightarrow$ (i).

For $n =0$, the following is an exact sequence:
$$0 \to {\rm K}_*(B ) \to {\rm K}_*(E) \to {\rm K}_*(A) \to 0.$$
We need to show this sequence is pure exact (see Definition \ref{group pure}).

For any $n\geq 1$,  consider the commutative diagram with exact rows as follows
$$
\xymatrixcolsep{3pc}
\xymatrix{
{\mathrm{K}_j(B)} \ar[d]_-{{\rm K}_j(\iota)} \ar[r]^-{\times n}
& {\mathrm{K}_j(B)} \ar[d]_-{{\rm K}_j(\iota)} \ar[r]^-{\rho_j^B}
& {\mathrm{K}_j(B; \mathbb{Z}_n)} \ar[d]_-{\mathrm{K}_j(\iota\,;\, \mathbb{Z}_n)} \ar[r]^-{}
& 
 {\mathrm{K}_{1-j}(B)} \ar[d]_-{{\rm K}_{1-j}(\iota)}
 \\
{\mathrm{K}_j(E)} \ar[r]_-{\times n}
& {\mathrm{K}_j(E)} \ar[r]_-{\rho_j^E}
& {\mathrm{K}_j(E; \mathbb{Z}_n)} \ar[r]_-{}
& 
{\mathrm{K}_{1-j}(E)},
}
$$
where $j=0,1$.

Since it is obvious that $n\times {\rm K}_j(\iota) (\mathrm{K}_j(B))\subset
{\rm K}_j(\iota)(\mathrm{K}_j(B))\cap
(n\times \mathrm{K}_j(E)),$
we only need to check
$$n\times {\rm K}_j(\iota)(\mathrm{K}_j(B))\supset
{\rm K}_j(\iota) (\mathrm{K}_j(B))\cap
(n\times \mathrm{K}_j(E)).$$

For any $\tilde{y}\in {\rm K}_j(\iota)(\mathrm{K}_j(B))\cap
(n\times \mathrm{K}_j(E))$,
there exists $x\in \mathrm{K}_j(B)$ and
$y\in\mathrm{K}_j(E)$
such that
$$
{\rm K}_j(\iota)(x)=\tilde{y}
\quad{\rm and}\quad
 ny=\tilde{y}.
$$
By the exactness of the second row, we have $\rho_j^E(\tilde{y})=0=\rho_j^E({\rm K}_j(\iota)(x))$. The commutativity implies
 $$\mathrm{K}_j(\iota\,;\, \mathbb{Z}_n)\circ \rho_j^B(x)=0.$$
By assumption,  $\mathrm{K}_j(\iota\,;\, \mathbb{Z}_n)$ is an injective map, which means  $\rho_j^B(x)=0$. By the exactness of the first row,
there exists an $\tilde{x}\in \mathrm{K}_j(B)$ such that
$n\tilde{x}=x.$ Now we have
$$
\tilde{y}={\rm K}_j(\iota)(x)={\rm K}_j(\iota)(n\tilde{x})
=n\times {\rm K}_j(\iota)(\tilde{x})\in n\times {\rm K}_j(\iota)(\mathrm{K}_j(B)).
$$
This concludes the proof.

\end{proof}

\begin{remark}\rm
Let $I$ be an ideal of  $A$. Then
$0\to I\xrightarrow{\iota} A\to A/I\to 0$ is an extension.
From Lemma \ref{k-pure lemma},
if $I$ is K-pure in $A$, it means that when we check the image of
$\underline{\mathrm{K}}(I)$ in $\underline{\mathrm{K}}(A)$ through $\underline{\mathrm{K}}(\iota)$, all the information is kept ($\underline{\mathrm{K}}(\iota)$ is injective);
on the contrary, if $I$ is not K-pure in $A$, the information contained in $\ker(\underline{\mathrm{K}}(\iota))$ will be lost when we only fucus on $\underline{\mathrm{K}}(A)$.
Particularly, for the non K-pure algebras $E_1,E_2$ in \ref{counter ex}, some key characterizations of $\underline{\mathrm{K}}(B_1)$ and $\underline{\mathrm{K}}(B_2)$
get lost in
$\underline{\mathrm{K}}(E_1)$ and $\underline{\mathrm{K}}(E_2)$, respectively.
\end{remark}

\begin{theorem}\label{k pure ideal and q}
If $A$ is a $\mathrm{K}$-pure ${\rm C}^*$-algebra, then for any ideal $I$ of $A$, both $I$ and $A/I$ are $\mathrm{K}$-pure ${\rm C}^*$-algebras.
\end{theorem}
\begin{proof}
Given any ideal $I'$ of $I$,  $I'$ is also an ideal of $A$. By assumption, both $I'$ and $I$ are $\mathrm{K}$-pure in $A$. Consider the following commutative diagram
$$
\xymatrixcolsep{3pc}
\xymatrix{
{\,\,\underline{\mathrm{K}}(I)\,\,} \ar[r]^-{\underline{\rm K}(\iota_{IA})}
& {\,\,\underline{\mathrm{K}}(A).\,\,}
 \\
{\,\,\underline{\mathrm{K}}(I')\,\,} \ar[ur]_-{\underline{\rm K}(\iota_{I'A})}\ar[u]^-{\underline{\rm K}(\iota_{I'I})}}
$$
Lemma \ref{k-pure lemma} (iii) implies both $\underline{\rm K}(\iota_{IA})$ and $\underline{\rm K}(\iota_{I'A})$ are injective, then $\underline{\rm K}(\iota_{I'I})$ is injective. Apply Lemma \ref{k-pure lemma} (iii) again, we have $I'$ is K-pure in $I$. Then $I$ is a K-pure ${\rm C}^*$-algebra.

The proof for $A/I$ is similar by using the surjectivity and Lemma \ref{k-pure lemma} (iv).

\end{proof}
\begin{corollary}\label{K pure transitivity}
 Let $A,B,C$ be ${\rm C}^*$-algebras. If $A$ is {\rm K}-pure in $B$, $B$ is {\rm K}-pure in $C$, then $A$ is {\rm K}-pure in $C$.
\end{corollary}
\begin{corollary}
If $A_1\xrightarrow{\psi_{1,2}} A_2\xrightarrow{\psi_{2,3}}  A_3\to \cdots$ is an inductive system of $\mathrm{K}$-pure ${\rm C}^*$-algebras with an inductive limit $A$,
then $A$ is $\mathrm{K}$-pure.
\end{corollary}
\begin{proof}
For each $n\geq 1,$ denote the induced map from
$A_n$ to $A$ by $\psi_n$. Given any ideal $I$ of $A$,
we have an inductive system
$$I_1\xrightarrow{\psi_{1,2}'} I_2\xrightarrow{\psi_{2,3}'}  I_3\to \cdots,$$
where $I_n=\psi_n^{-1}(I)$ is an ideal of $A_n$, $\psi_{n,n+1}'$ is the restriction map. Then, $I=\lim\limits_{\longrightarrow} (I_n,\psi_{n,n+1}')$.
Since each $I_n$ is K-pure in $A_n$, by Lemma \ref{k-pure lemma} (iii),  it is routine to check that $I$ is K-pure in $A$ from the fact that both
$\underline{\rm K}(I)$ and $\underline{\rm K}(A)$ are the algebraic inductive limits. Then we conclude that $A$ is K-pure.

\end{proof}

\begin{theorem}\label{k-pure ext k-pure}
Let $A,B$ be $\mathrm{K}$-pure ${\rm C}^*$-algebras. Then for any
$\mathrm{K}$-pure extension
$$
0\to B\xrightarrow{\iota} E\xrightarrow{\pi} A\to 0,
$$
we have $E$ is a $\mathrm{K}$-pure ${\rm C}^*$-algebra.
\end{theorem}

\begin{proof}
We identify $B$ as the its image in $E$ and identify $A$ as $E/B$. Take any ideal $I$ of $E$. For any $n\geq0$, $j=0,1,$ we have the following commutative diagram with exact rows
induced by the natural inclusion maps
$$
\xymatrixcolsep{0.4pc}
\xymatrix{
{\mathrm{K}_{1-j}(I_1;\mathbb{Z}_n)} \ar[d]_-{\cong}\ar[r]^-{}
& {\mathrm{K}_j(I\cap B;\mathbb{Z}_n)} \ar[d]_-{\mathrm{K}_j(\iota_1;\mathbb{Z}_n)} \ar[r]^-{}
& {\mathrm{K}_j(I;\mathbb{Z}_n)} \ar[d]_-{\mathrm{K}_j(\iota_2;\mathbb{Z}_n)} \ar[r]_-{}
& {\mathrm{K}_j(I_1;\mathbb{Z}_n)} \ar[d]_-{\cong} \ar[r]_-{}
& {\mathrm{K}_{1-j}(I\cap B;\mathbb{Z}_n)} \ar[d]_-{\mathrm{K}_{1-j}(\iota_1;\mathbb{Z}_n)}
\\
 {\mathrm{K}_{1-j}(I_2;\mathbb{Z}_n)}\ar[r]^-{}
& {\mathrm{K}_j(B;\mathbb{Z}_n)} \ar[r]^-{}
& {\mathrm{K}_j(I+B;\mathbb{Z}_n)} \ar[r]_-{}
& {\mathrm{K}_j(I_2;\mathbb{Z}_n)} \ar[r]_-{}
& 
{\mathrm{K}_{1-j}(B;\mathbb{Z}_n)},
}
$$
where $I_1:=I/(I\cap B)$, $I_2:=(I+B)/B$, $\iota_1 :I\cap B\to B$ and  $\iota_2:I\rightarrow I+B$.

Since $B$ is K-pure, then $I\cap B$ is K-pure in $B$. By Lemma \ref{k-pure lemma}, both $\mathrm{K}_0(\iota_1;\mathbb{Z}_n)$ and $\mathrm{K}_{1}(\iota_1;\mathbb{Z}_n)$ are injective.
Using the weak version of five lemma, we have $\mathrm{K}_j(\iota_2;\mathbb{Z}_n)$ is injective.
By Lemma \ref{k-pure lemma} again,  $I$ is K-pure in $I+B$.

Consider the following commutative diagram:
$$
\xymatrixcolsep{2pc}
\xymatrix{
{\,\,0\,\,} \ar[r]^-{}
& {\,\,B\,\,} \ar[d]_-{} \ar[r]^-{\iota}
& {\,\,E\,\,} \ar[d]_-{{\rm id}} \ar[r]^-{\pi}
& {\,\,A\,\,} \ar[d]_-{\pi_{AA'}} \ar[r]^-{}
& {\,\,0\,\,} \\
{\,\,0\,\,} \ar[r]^-{}
& {\,\,I+B\,\,} \ar[r]_-{}
& {\,\,E \,\,} \ar[r]_-{\pi_{EA'}}
& {\,\,A'\,\,} \ar[r]_-{}
& {\,\,0\,\,},}
$$
where $A':=E/(I+B)\cong A/I_2$ and $\pi_{AA'}$, $\pi_{EA'}$ are the relative  quotient maps, respectively.

Since $B$ is K-pure in $E$ and $I_2$ is K-pure in $A$, by Lemma \ref{k-pure lemma}, for any $n\geq0$, $j=0,1,$ we have both $\mathrm{K}_j(\pi;\mathbb{Z}_n)$ and $\mathrm{K}_j(\pi_{AA'};\mathbb{Z}_n)$ are surjective. Then the composed map $\mathrm{K}_j(\pi_{EA'};\mathbb{Z}_n)$ is also surjective.
Still by Lemma \ref{k-pure lemma}, we have $I+B$ is K-pure in $E$.
From Corollary \ref{K pure transitivity}, $I$ is K-pure in $E$.
In general, $E$ is K-pure.

\end{proof}

\section{Classification: K-pure extension}
In this section, we give our classification theorems for the class of K-pure algebras arising from K-pure extensions. The main result is Theorem \ref{ab1b2}.

We say a class $\mathcal{R}$ of ${\rm C}^*$-algebras is a  $\mathcal{DG}$-class, if  $\mathcal{R}$ is a subclass of $\mathcal{N}$ consisting of some K-pure, nuclear, separable ${\rm C}^*$-algebras of real rank zero and stable rank one, whose ${\rm K}_0$-groups are weakly unperforated and $\mathcal{R}$ can be classified by scaled  total K-theory
with Dadarlat-Gong order. (Here, we require that each isomorphism between invariants can be lifted to an isomorphism between algebras.)

Thus, for a $\mathcal{DG}$-class,
UCT, UMCT and  Theorem \ref{strong wei} will be satisfied in the corresponding cases.

Given an extension of abelian groups
   $$e:\quad 0\rightarrow  K \xrightarrow{\iota} G \xrightarrow{\pi} H\rightarrow 0$$
it is easily seen that the following are equivalent:

(i) $e$ is pure exact, i.e., $
n\times \iota(K)=\iota(K)\cap (n\times G)$   for every $n\in\mathbb{N}$.

(ii) $\iota(K)$ is a relatively divisible subgroup of $G$,
i.e., for every $n\in \mathbb{N}$, $h\in \iota(K)$ is divisible by $n$ in $\iota(K)$ if
$h$ is divisible by $n$ in $G$ .

Note that if $A$ is separable ${\rm C}^*$-algebra of  stable rank one and real rank zero,
there is a natural isomorphism $\phi:\,{\rm Lat}(A)\cong {\rm Lat}({\rm K}_*(A))$, where
${\rm Lat}(A)$ is the lattice of ideals of $A$,
${\rm Lat}({\rm K}_*(A))$ is the lattice of order ideals of $({\rm K}_*(A),{\rm K}_*^+(A))$, and
$\phi(I)=({\rm K}_*(I),{\rm K}_*^+(I))$.
Combining these with \cite[2.1 (iii)]{Ell} and
\cite[Paragraph 3 of Theorem 4.28]{EG}, we obtain:

\begin{proposition}\label{ell propo}
Let $A$ be a separable ${\rm C}^*$-algebra of stable rank one and real rank zero, then the following are equivalent:

(i) $(\mathrm{K}_*(A),\mathrm{K}_*^+(A))$ is weakly unperforated (see Definition \ref{0weakly un});

(ii) $(\mathrm{K}_0(A),\mathrm{K}_0^+(A))$ is weakly unperforated and $A$ is {\rm K}-pure.
\end{proposition}

It is already well-known that
the ${\rm K}_*$-group of
an A$\mathcal{HD}$ algebra of real rank zero is weakly unperforated, as the ${\rm K}_*$-groups of the building blocks of an A$\mathcal{HD}$ algebra are all weakly unperforated.
Combining with the following theorem, we have the class of A$\mathcal{HD}$ algebras of real rank zero is a $\mathcal{DG}$-class.
\begin{theorem}{\rm (}{\cite[Theorem 9.1, Remark 9.3]{DG}}{\rm )}
Let $A, B$ be two A$\mathcal{HD}$ algebras of real rank zero. 
Suppose that there is an isomorphism of ordered scaled
groups
$$
\alpha\,:\,({\rm\underline{\mathrm{K}}}(A),{\rm\underline{K}}(A)_+,\Sigma A)\rightarrow ({\rm\underline{\mathrm{K}}}(B),{\rm\underline{K}}(A)_+,\Sigma B)
$$
which preserves the action of the Bockstein operations, i.e., $\alpha$ is an ordered scaled $\Lambda$-isomorphism. Then there is a
$*$-isomorphism  $\varphi :A\to B$ with $\underline{\mathrm{K}}(\varphi)=\alpha$.
\end{theorem}

\begin{proposition}\label{simple k* to k total}
Let $A_1,A_2$ be simple algebras in $\mathcal{N}$. Then for any
$$
\alpha^*:\,(\mathrm{K}_*(A_1),\mathrm{K}_*^+(A_1))\cong (\mathrm{K}_*(A_2),\mathrm{K}_*^+(A_2)),
$$
there exists an
$$
\underline{\alpha}:\,(\underline{\mathrm{K}}(A_1),\underline{\mathrm{K}}(A_1)_+)\cong (\underline{\mathrm{K}}(A_2),\underline{\mathrm{K}}(A_2)_+)
$$
whose restriction is $\alpha^*$.
\end{proposition}
\begin{proof}
By UCT (Theorem \ref{UCT}), there is a KK-equivalence $\alpha\in \mathrm{K}\mathrm{K}(A_1,A_2)$ inducing the given
$$
\alpha^*:\,(\mathrm{K}_*(A_1),\mathrm{K}_*^+(A_1))\cong (\mathrm{K}_*(A_2),\mathrm{K}_*^+(A_2)).
$$
By UMCT (Theorem \ref{UMCT}) (see also \ref{def k-total}), denote $\underline{\alpha}$ the natural image of $\alpha$ in $ {\rm Hom}_\Lambda(\underline{\mathrm{K}}(A_1),\underline{\mathrm{K}}(A_2))$, which is a $\Lambda$-isomorphism.

Since $A_1$ and $A_2$ are simple, the Dadarlat-Gong orders only depend on the positive cones $\mathrm{K}_0^+(A_1)$ and $\mathrm{K}_0^+(A_2)$, which means $\underline{\alpha}$ is an order-preserving  $\Lambda$-isomorphism.

\end{proof}
\begin{lemma}\label{alpha0}
Let $A_1, A_2, B_1,B_2\in\mathcal{R}$, where $\mathcal{R}$ is a $\mathcal{DG}$-class of ${\rm C}^*$-algebras. Suppose $A_1, A_2$ are unital simple, $B_1, B_2$ are stable.
Given two unital essential extensions of ${\rm C}^*$-algebras
$$
e_i\,\,:\,\,0\to B_i \xrightarrow{\iota_i} (E_i,1_{E_i})\xrightarrow{\pi_i} (A_i,1_{A_i})\to 0,\quad i=1,2
$$
with trivial index maps,
if we have an order-preserving  isomorphism $$\alpha:\,(\mathrm{K}_0(E_1),\mathrm{K}_0^+(E_1),[1_{E_1}])\to (\mathrm{K}_0(E_2),\mathrm{K}_0^+(E_2),[1_{E_2}]),$$
then we have a natural commutative diagram as follows
$$
\xymatrixcolsep{2pc}
\xymatrix{
{\,\,0\,\,} \ar[r]^-{}
& {\,\,\mathrm{K}_0(B_1)\,\,} \ar[d]_-{\alpha_0} \ar[r]^-{\mathrm{K}_0(\iota_1)}
& {\,\,(\mathrm{K}_0(E_1),[1_{E_1}])\,\,} \ar[d]_-{\alpha} \ar[r]^-{\mathrm{K}_0(\pi_1)}
& {\,\,(\mathrm{K}_0(A_1),[1_{A_1}])\,\,} \ar[d]_-{\alpha_1} \ar[r]^-{}
& {\,\,0\,\,} \\
{\,\,0\,\,} \ar[r]^-{}
& {\,\,\mathrm{K}_0(B_2)\,\,} \ar[r]_-{\mathrm{K}_0(\iota_2)}
& {\,\,(\mathrm{K}_0(E_2),[1_{E_2}]) \,\,} \ar[r]_-{\mathrm{K}_0(\pi_2)}
& {\,\,(\mathrm{K}_0(A_2) ,[1_{A_2}])\,\,} \ar[r]_-{}
& {\,\,0\,\,},}
$$
where $\alpha_0,\alpha_1$ are both  scaled  order-preserving  isomorphisms.
\end{lemma}
\begin{proof}
We first prove that for any projection $p\in B_1$, we have
$$\mathrm{K}_0(\pi_2)\circ\alpha\circ \mathrm{K}_0(\iota_1)([p])=0.$$

Suppose
$\mathrm{K}_0(\pi_2)\circ\alpha\circ \mathrm{K}_0(\iota_1)([p])\neq0.$
Since $A_2$ is simple and unital, there exists a large enough integer $n$ satisfying
$$
3\times [1_{A_2}]\leq n\times \mathrm{K}_0(\pi_2)\circ\alpha\circ \mathrm{K}_0(\iota_1)([p]) \,\,\,\,{\rm in}\,\, \mathrm{K}_0(A_2).
$$
Note that $E_2$ is of real rank zero, by  Proposition \ref{lin inj} (i)(c), there exists a nonzero projection
$s\in E_2\otimes\mathcal{K}$ such that
$$\mathrm{K}_0(\pi_2)([s])=
n\times \mathrm{K}_0(\pi_2)\circ\alpha\circ \mathrm{K}_0(\iota_1)([p])-2\times [1_{A_2}]\in \mathrm{K}_0^+({A_2}).
$$
Hence, $$
\mathrm{K}_0(\pi_2)(n\times \alpha\circ \mathrm{K}_0(\iota_1)([p])- 2[1_{E_2}])=\mathrm{K}_0(\pi_2)([s]),$$
by the exactness of the second row, there exist projections $f,g\in B_2$ such that
$$
[s]=n\times \alpha\circ \mathrm{K}_0(\iota_1)([p])- 2[1_{E_2}]+\mathrm{K}_0(\iota_2)([f])- \mathrm{K}_0(\iota_2)([g])\in \mathrm{K}_0^+(E_2),
$$
Since the extension $e_1$ is  unital 
and $B_1$ is stable, we have
$n\times \mathrm{K}_0(\iota_1)([p])\leq [1_{E_1}]$. Then
$$n\times\alpha\circ \mathrm{K}_0(\iota_1)([p])\leq \alpha( [1_{E_1}])=[1_{E_2}].$$

As the extension $e_2$ is  unital 
and $B_2$ is stable, we have
$$- [1_{E_2}]+\mathrm{K}_0(\iota_2)([f])\in -\mathrm{K}_0^+(E_2).$$
Then
$$n\times\alpha\circ \mathrm{K}_0(\iota_1)([p])- 2[1_{E_2}]+\mathrm{K}_0(\iota_2)[f]\in -\mathrm{K}_0^+(E_2),$$
$$
[s]=n\times \alpha\circ \mathrm{K}_0(\iota_1)([p])- 2[1_{E_2}]+
\mathrm{K}_0(\iota_2)([f])- \mathrm{K}_0(\iota_2)([g])\in  -\mathrm{K}_0^+(E_2).
$$
Then we have $[s]=0$, and hence, $$n\times \mathrm{K}_0(\pi_2)\circ\alpha\circ \mathrm{K}_0(\iota_1)([p])=2\times[1_{A_2}],$$
which leads a contradiction.

Now from the exactness of the second row, there exists a unique element $[t]\in \mathrm{K}_0^+(B_2)$ such that $\mathrm{K}_0(\iota_2)([t])=\alpha\circ \mathrm{K}_0(\iota_1)([p])$. Denote $\alpha_0$ the Grothendieck map of the corresponding map  $$[p]\mapsto [t].$$
Then $\alpha_0$ is an injective order-preserving homomorphism from $\mathrm{K}_0(B_1)$ to $\mathrm{K}_0(B_2)$. Note that $\alpha$ is an isomorphism, with a same procedure for $\alpha^{-1}$, we have $\alpha_0$ is an order-preserving isomorphism.

At last, we point out that Proposition \ref{lin inj} (i)(c) guarantees the induced map $\alpha_1$ is also an order-preserving isomorphism.

\end{proof}
\begin{corollary}\label{* AND Total}
Let $A_1, A_2, B_1,B_2\in\mathcal{R}$, where $\mathcal{R}$ is a $\mathcal{DG}$-class of $C^*$-algebras. Suppose $A_1, A_2$ are unital simple, $B_1, B_2$ are stable.
Given two $\mathrm{K}$-pure unital essential  extensions
$$
e_i\,\,:\,\,0\to B_i \xrightarrow{\iota_i} (E_i,1_{E_i})\xrightarrow{\pi_i} (A_i,1_{A_i})\to 0,\quad i=1,2,
$$
 if we have an ordered scaled $\Lambda$-isomorphism $$\alpha:\,(\underline{\mathrm{K}}(E_1),\underline{\mathrm{K}}(E_1)_+,[1_{E_1}])\to (\underline{\mathrm{K}}(E_2),\underline{\mathrm{K}}(E_2)_+,[1_{E_2}]),$$
then we have a natural commutative diagram with exact rows as follows
$$
\xymatrixcolsep{2pc}
\xymatrix{
{\,\,0\,\,} \ar[r]^-{}
& {\,\,\underline{\mathrm{K}}(B_1)\,\,} \ar[d]_-{\alpha_0} \ar[r]^-{\underline{\rm K}(\iota_1)}
& {\,\,(\underline{\mathrm{K}}(E_1),[1_{E_1}])\,\,} \ar[d]_-{\alpha} \ar[r]^-{\underline{\rm K}(\pi_1)}
& {\,\,(\underline{\mathrm{K}}(A_1),[1_{A}])\,\,}  \ar[d]_-{\alpha_1}\ar[r]^-{}
& {\,\,0\,\,} \\
{\,\,0\,\,} \ar[r]^-{}
& {\,\,\underline{\mathrm{K}}(B_2)\,\,} \ar[r]_-{\underline{\rm K}(\iota_2)}
& {\,\,(\underline{\mathrm{K}}(E_2),[1_{E_2}]) \,\,} \ar[r]_-{\underline{\rm K}(\pi_2)}
& {\,\,(\underline{\mathrm{K}}(A_2),[1_{A}]) \,\,} \ar[r]_-{}
& {\,\,0\,\,},}
$$
where
$\alpha_0,\alpha_1$ are also ordered scaled $\Lambda$-isomorphisms.

\end{corollary}
\begin{proof}
By assumption, we have the following diagram with exact rows:
$$
\xymatrixcolsep{2pc}
\xymatrix{
{\,\,0\,\,} \ar[r]^-{}
& {\,\,\underline{\mathrm{K}}(B_1)\,\,} \ar[r]^-{\underline{\rm K}(\iota_1)}
& {\,\,(\underline{\mathrm{K}}(E_1),[1_{E_1}])\,\,} \ar[d]_-{\alpha} \ar[r]^-{\underline{\rm K}(\pi_1)}
& {\,\,(\underline{\mathrm{K}}(A_1),[1_{A}])\,\,}  \ar[r]^-{}
& {\,\,0\,\,} \\
{\,\,0\,\,} \ar[r]^-{}
& {\,\,\underline{\mathrm{K}}(B_2)\,\,} \ar[r]_-{\underline{\rm K}(\iota_2)}
& {\,\,(\underline{\mathrm{K}}(E_2),[1_{E_2}]) \,\,} \ar[r]_-{\underline{\rm K}(\pi_2)}
& {\,\,(\underline{\mathrm{K}}(A_2),[1_{A}]) \,\,} \ar[r]_-{}
& {\,\,0\,\,},}
$$
where $\alpha$ is an ordered  scaled    $\Lambda$-isomorphism.

For any $(x,\bar{x})\in \underline{\mathrm{K}}(B_1)_+$ with $x\in \mathrm{K}_0^+(B_1)$, by Lemma \ref{alpha0},  we have
$$\underline{\rm K}(\pi_2)\circ\alpha\circ \underline{\rm K}(\iota_1)(x,0)=(0,0).$$
By Dadarlat-Gong order \ref{dg order}, we have
$$
(x,\bar{x})\in \underline{\mathrm{K}}(I_x\mid B_1).
$$
As both $\underline{\rm K}(\iota_1)$ and $\alpha$ are order-preserving maps, we have
$$
(y,\bar{y})=:\alpha\circ\underline{\rm K}(\iota_1)(x,\bar{x})\in \underline{\mathrm{K}}( E_2)_+.
$$
As $\underline{\rm K}(\pi_2)$ is also an order-preserving map, we have
$$
\underline{\rm K}(\pi_2)(y,\overline{y})=(0,w)\in \underline{\mathrm{K}}(A_2)_+.
$$
Then we have $w=0$.

From injectivity of $\underline{\rm K}(\iota_2)$, denote
$$
(z,\overline{z})=:\underline{\rm K}(\iota_2)^{-1}(y,\overline{y}).
$$
Note that $z$, $y$ are both positive elements in $\mathrm{K}_0$-groups, denote $I_z,\,I_y$ the ideals generated by
$z$, $y$ in $B_2$ and $E_2$, respectively. By assumption, $\iota_2$ is injective and $\iota_2(I_z)$ is an ideal in $E_2$, then the fact ${\rm K}_0(\iota_2)(z)=y$ implies $\iota_2(I_z)=I_y$, we have ${\iota_2}_{I_zI_y}=:\iota_2|_{I_z\to I_y}$ is an isomorphism.

Now we have the following commutative diagram
$$
\xymatrixcolsep{2pc}
\xymatrix{
{\,\,\underline{\mathrm{K}}(I_z)\,\,} \ar[d]_-{\underline{\mathrm{K}}(\iota_{I_zB_2})} \ar[r]^-{\underline{\mathrm{K}}({\iota_2}_{I_zI_y})}
& {\,\,\underline{\mathrm{K}}(I_y)\,\,} \ar[d]^-{\underline{\mathrm{K}}(\iota_{I_yE_2})} \\
{\,\,\underline{\mathrm{K}}(B_2) \,\,} \ar[r]_-{\underline{\mathrm{K}}({\iota_2})}
& {\,\,\underline{\mathrm{K}}(E_2) \,\,},}
$$
where $\iota_{I_zB_2}, \iota_{I_yE_2}$ are the natural inclusions.

Since $(y,\overline{y})\in \underline{\mathrm{K}}(I_y\mid E_2)$, then
there exist $(0,\widehat{y})\in \underline{\mathrm{K}}(I_y)$, $(0,\widehat{z}) \in \underline{\mathrm{K}}(I_z)$  such that
$$
\underline{\mathrm{K}}(\iota_{I_yE_2})(0,\widehat{y})=(0,\overline{y})
\quad {\rm and}\quad
(0,\widehat{z})=\underline{\mathrm{K}}({\iota_2}_{I_zI_y})^{-1}(0,\widehat{y}).
$$
From the commutativity, we have
$$
\underline{\mathrm{K}}(\iota_2)\circ\underline{\mathrm{K}}(\iota_{I_zB_2})(0,\widehat{z})=(0,\overline{y}).
$$
By the injectivity of $\underline{\mathrm{K}}(\iota_2)$, $
\underline{\mathrm{K}}(\iota_{I_zB_2})(0,\widehat{z})=(0,\overline{z}),
$
that is, $(z,\bar{z})\in \underline{\mathrm{K}}(I_z| B_2)$, and hence,
$
(z,\bar{z})\in\underline{\mathrm{K}}(B_2)_+.
$

With a similar argument in Lemma \ref{alpha0} and apply Theorem \ref{ordertotal} , we have that the Grothendieck map $\alpha_0$ of the following map
$$
(x,\overline{x})\mapsto(z,\overline{z})
$$
is an order-preserving  isomorphism from $\underline{\mathrm{K}}(B_1)$ to $\underline{\mathrm{K}}(B_2)$. Note that $A_1$, $A_2$ is simple, the order on the total K-theory is determined by the positive cone of  $\mathrm{K}_0$-group. Similarly, we can obtain the order-preserving isomorphism  $\alpha_1$ as expected. 

Moreover, since both the extensions  $e_1, e_2$ are K-pure, by Theorem \ref{k-pure ext k-pure}, we have $E_1,E_2$ are both K-pure.
Then by Lemma \ref{k-pure lemma} (iii), for each $i=1,2,$
we have
$\underline{\mathrm{K}}(B_i)$ is regarded as
the sub $\Lambda$-module $\underline{\mathrm{K}}(E_i)$ through
$\underline{\mathrm{K}}(\iota_i)$,
and $\alpha_0$ is indeed an ordered scaled  $\Lambda$-isomorphism; so does the induced map $\alpha_1$.

\end{proof}

\begin{definition}
Let $A$, $B$ be ${\rm C}^*$-algebras. Two extensions $e_1$ and $e_2$ of $A$ by $B$ are called weakly isomorphic,
denoted by $e_1\simeq_w e_2$, if there exist isomorphisms $\psi_0,\psi,\psi_1$ making the following diagram commute:
$$
\xymatrixcolsep{2pc}
\xymatrix{
{\,\,e_1:\quad0\,\,} \ar[r]^-{}
& {\,\,B\,\,} \ar[d]_-{\psi_0} \ar[r]^-{}
& {\,\,E_1\,\,} \ar[d]_-{\psi} \ar[r]^-{}
& {\,\,A\,\,} \ar[d]_-{\psi_1} \ar[r]^-{}
& {\,\,0\,\,} \\
{\,\,e_2:\quad0\,\,} \ar[r]^-{}
& {\,\,B\,\,} \ar[r]_-{}
& {\,\,E_2 \,\,} \ar[r]_-{}
& {\,\,A \,\,} \ar[r]_-{}
& {\,\,0\,\,}.}
$$
\end{definition}
\begin{theorem}\label{4 dengjia}
Let $A, B \in\mathcal{R}$, where $\mathcal{R}$ is a $\mathcal{DG}$-class of ${\rm C}^*$-algebras. Assume that $A$ is unital simple, $B$ is stable.
Given two unital essential $\mathrm{K}$-pure extensions of ${\rm C}^*$-algebras
$$
e_i\,\,:\,\,0\to B \xrightarrow{\iota_i} E_i\xrightarrow{\pi_i}A\to 0,\quad i=1,2.
$$
The following statements are equivalent:

{\rm (i)} $E_1\cong E_2$;

{\rm (ii)} $e_1$ and $e_2$ are weakly isomorphic;

{\rm (iii)} $(\underline{\mathrm{K}}(E_1),\underline{\mathrm{K}}(E_1)_+,[1_{E_1}])\cong
(\underline{\mathrm{K}}(E_2),\underline{\mathrm{K}}(E_2)_+,[1_{E_2}])$.
\end{theorem}
\begin{proof}

{\rm (i)} $\Rightarrow$ {\rm (ii)}: Given an isomorphism $\psi:\, E_1\to E_2$,
by the proof of Lemma \ref{alpha0}, for any projection $p\in B$,
we have
$${\rm K}_0(\pi_2)\circ\alpha\circ {\rm K}_0(\iota_1)([p])=[\pi_2\circ\psi\circ \iota_1(p)]=0.$$

Since $A$ has stable rank one, then $\pi_2\circ\psi\circ \iota_1(p)=0$. By exactness,
there exists a unique projection  $p'\in B$, such that $\iota_2(p')=\psi\circ \iota_1(p)$, i.e.,
for $\psi\circ \iota_1(p)$, we may define $p'=\iota_2^{-1}(\psi\circ \iota_1(p))$.
Then
$\iota_2^{-1}\circ\psi\circ \iota_1$ is a bijection between projections in $B$. Since $B$ is generated by its projections, we have
$\iota_2^{-1}\circ\psi\circ \iota_1$ is an automorphism on $B$ and we
denote it by $\psi_0$. Then we get an induced automorphism $\psi_1$ on $A$ and the following is commutative.
$$
\xymatrixcolsep{2pc}
\xymatrix{
{\,\,0\,\,} \ar[r]^-{}
& {\,\,B\,\,} \ar[d]_-{\psi_0} \ar[r]^-{}
& {\,\,E_1\,\,} \ar[d]_-{\psi} \ar[r]^-{}
& {\,\,A\,\,} \ar[d]_-{\psi_1} \ar[r]^-{}
& {\,\,0\,\,\,} \\
{\,\,0\,\,} \ar[r]^-{}
& {\,\,B\,\,} \ar[r]_-{}
& {\,\,E_2 \,\,} \ar[r]_-{}
& {\,\,A \,\,} \ar[r]_-{}
& {\,\,0\,\,}}
$$

 {\rm (ii)} $\Rightarrow$ {\rm (iii)}: Trivial.

{\rm (iii)} $\Rightarrow$ {\rm (i)}: Given a unital order-preserving $\Lambda$-isomorphism $\alpha$. 
$$\xymatrixcolsep{2pc}
\xymatrix{
{\,\,0\,\,} \ar[r]^-{}
& {\,\,\underline{\mathrm{K}}(B)\,\,} \ar[r]^-{\underline{\rm K}(\iota_1)}
& {\,\,(\underline{\mathrm{K}}(E_1),[1_{E_1}])\,\,} \ar[d]_-{\alpha} \ar[r]^-{\underline{\rm K}(\pi_1)}
& {\,\,(\underline{\mathrm{K}}(A),[1_{A}])\,\,} \ar[r]^-{}
& {\,\,0\,\,} \\
{\,\,0\,\,} \ar[r]^-{}
& {\,\,\underline{\mathrm{K}}(B)\,\,} \ar[r]_-{\underline{\rm K}(\iota_2)}
& {\,\,(\underline{\mathrm{K}}(E_2),[1_{E_2}]) \,\,} \ar[r]_-{\underline{\rm K}(\pi_2)}
& {\,\,(\underline{\mathrm{K}}(A),[1_{A}]) \,\,} \ar[r]_-{}
& {\,\,0\,\,}}
$$
By Corollary \ref{* AND Total} and definition of $\mathcal{DG}$ class, we have two homomorphisms $\psi_0,\psi_1$  
such that the following diagram is commutative with exact rows
$$\xymatrixcolsep{2pc}
\xymatrix{
{\,\,0\,\,} \ar[r]^-{}
& {\,\,\underline{\mathrm{K}}(B)\,\,} \ar[d]_-{\underline{\mathrm{K}}({\psi_0})}\ar[r]^-{\underline{\rm K}(\iota_1)}
& {\,\,(\underline{\mathrm{K}}(E_1),[1_{E_1}])\,\,} \ar[d]_-{\alpha} \ar[r]^-{\underline{\rm K}(\pi_1)}
& {\,\,(\underline{\mathrm{K}}(A),[1_{A}])\,\,} \ar[d]_-{\underline{\rm K}(\psi_1)} \ar[r]^-{}
& {\,\,0\,\,} \\
{\,\,0\,\,} \ar[r]^-{}
& {\,\,\underline{\mathrm{K}}(B)\,\,} \ar[r]_-{\underline{\rm K}(\iota_2)}
& {\,\,(\underline{\mathrm{K}}(E_2),[1_{E_2}]) \,\,} \ar[r]_-{\underline{\rm K}(\pi_2)}
& {\,\,(\underline{\mathrm{K}}(A),[1_{A}]) \,\,} \ar[r]_-{}
& {\,\,0\,\,},}
$$
which can be transformed  into
$$
\xymatrixcolsep{3pc}
\xymatrix{
{\,\,0\,\,} \ar[r]^-{}
& {\,\,\underline{\mathrm{K}}(B)\,\,} \ar[d]_-{{\rm id}} \ar[r]^-{\underline{\rm K}(\iota_1)}
& {\,\,(\underline{\mathrm{K}}(E_1),[1_{E_1}])\,\,} \ar[d]_-{\alpha} \ar[r]^-{\underline{\rm K}(\psi_1\circ{\pi_1})}
& {\,\,(\underline{\mathrm{K}}(A),[1_A])\,\,} \ar[d]_-{{\rm id}} \ar[r]^-{}
& {\,\,0\,\,} \\
{\,\,0\,\,} \ar[r]^-{}
& {\,\,\underline{\mathrm{K}}(B)\,\,} \ar[r]_-{ \underline{\rm K}({\iota_2}\circ {\psi_0})}
& {\,\,(\underline{\mathrm{K}}(E_2),[1_{E_2}]) \,\,} \ar[r]_-{\underline{\rm K}(\pi_2)}
& {\,\,(\underline{\mathrm{K}}(A),[1_A]) \,\,} \ar[r]_-{}
& {\,\,0\,\,}.}
$$
Restricting on $\mathrm{K}_*$,
we have
$$
\xymatrixcolsep{2pc}
\xymatrix{
{\,\,0\,\,} \ar[r]^-{}
& {\,\,\mathrm{K}_*(B)\,\,} \ar[d]_-{{\rm id}} \ar[r]^-{\mathrm{K}_*(\iota_1)}
& {\,\,(\mathrm{K}_*(E_1),[1_{E_1}])\,\,} \ar[d]_-{\alpha^*} \ar[r]^-{\mathrm{K}_* (\psi_1\circ\pi_1)}
& {\,\,(\mathrm{K}_*(A),[1_{A}])\,\,} \ar[d]_-{{\rm id}} \ar[r]^-{}
& {\,\,0\,\,} \\
{\,\,0\,\,} \ar[r]^-{}
& {\,\,\mathrm{K}_*(B)\,\,} \ar[r]_-{ \mathrm{K}_*(\iota_2\circ\psi_0)}
& {\,\,(\mathrm{K}_*(E_2),[1_{E_2}]) \,\,} \ar[r]_-{\mathrm{K}_*(\pi_2)}
& {\,\,(\mathrm{K}_*(A),[1_{A}]) \,\,} \ar[r]_-{}
& {\,\,0\,\,},}
$$
where $\alpha^*$ is the restriction map of $\alpha$.

Then by Theorem \ref{strong wei}, the unital essential extensions with trivial index maps
$$
0\to B \xrightarrow{\iota_1} E_1\xrightarrow{\psi_1\circ \pi_1} A\to 0
$$
and
$$
0\to B \xrightarrow{\iota_2\circ \psi_0} E_2\xrightarrow{\pi_2} A\to 0
$$
are strongly unitarily equivalent. By Proposition \ref{strong tui cong}, we have $E_1\cong E_2$.

\end{proof}
Now we raise the following classification theorem for the class of ${\rm C}^*$-algebras obtained from K-pure extensions in terms of total K-theory.
\begin{theorem}\label{ab1b2}
Let $A_1, A_2, B_1,B_2\in\mathcal{R}$, where $\mathcal{R}$ is a $\mathcal{DG}$-class of $C^*$-algebras. Suppose $A_1, A_2$ are  unital simple, $B_1, B_2$ are stable.
Given two unital essential {\rm K}-pure extensions of ${\rm C}^*$-algebras
$$
0\to B_i \xrightarrow{\iota_i} E_i\xrightarrow{\pi_i} A_i\to 0,\quad i=1,2.
$$
We have $E_1\cong E_2$ iff
$(\underline{\mathrm{K}}(E_1),\underline{\mathrm{K}}(E_1)_+,[1_{E_1}])\cong
(\underline{\mathrm{K}}(E_2),\underline{\mathrm{K}}(E_2)_+,[1_{E_2}])$.
\end{theorem}
\begin{proof}Suppose we have $\alpha:\,(\underline{\mathrm{K}}(E_1),\underline{\mathrm{K}}(E_1)_+,[1_{E_1}])\cong
(\underline{\mathrm{K}}(E_2),\underline{\mathrm{K}}(E_2)_+,[1_{E_2}])$, by Corollary \ref{* AND Total}, we have the following commutative diagram with exact rows
$$\xymatrixcolsep{2pc}
\xymatrix{
{\,\,0\,\,} \ar[r]^-{}
& {\,\,\underline{\mathrm{K}}(B_1)\,\,} \ar[d]_-{\alpha_0}\ar[r]^-{\underline{\rm K}(\iota_1)}
& {\,\,(\underline{\mathrm{K}}(E_1),[1_{E_1}])\,\,} \ar[d]_-{\alpha} \ar[r]^-{\underline{\rm K}(\pi_1)}
& {\,\,(\underline{\mathrm{K}}(A_1),[1_{A_1}])\,\,} \ar[d]_-{\alpha_1} \ar[r]^-{}
& {\,\,0\,\,} \\
{\,\,0\,\,} \ar[r]^-{}
& {\,\,\underline{\mathrm{K}}(B_2)\,\,} \ar[r]_-{\underline{\rm K}(\iota_2)}
& {\,\,(\underline{\mathrm{K}}(E_2),[1_{E_2}]) \,\,} \ar[r]_-{\underline{\rm K}(\pi_2)}
& {\,\,(\underline{\mathrm{K}}(A_2),[1_{A_2}]) \,\,} \ar[r]_-{}
& {\,\,0\,\,},}
$$
where both $\alpha_0,\alpha_1$ are scaled order-preserving $\Lambda$-isomorphism.
Lifting $\alpha_0,\alpha_1$ to $\psi_0,\psi_1$, we have
$$
\xymatrixcolsep{2pc}
\xymatrix{
{\,\,0\,\,} \ar[r]^-{}
& {\,\,\underline{\mathrm{K}}(B_1)\,\,} \ar[d]_-{{\rm id}} \ar[r]^-{\underline{\rm K}(\iota_1)}
& {\,\,(\underline{\mathrm{K}}(E_1),[1_{E_1}])\,\,} \ar[d]_-{\alpha} \ar[r]^-{\underline{\rm K}({\psi}_1\circ{\pi_1})}
& {\,\,(\underline{\mathrm{K}}(A_2),[1_{A_2}])\,\,} \ar[d]_-{{\rm id}} \ar[r]^-{}
& {\,\,0\,\,} \\
{\,\,0\,\,} \ar[r]^-{}
& {\,\,\underline{\mathrm{K}}(B_1)\,\,} \ar[r]_-{ \underline{\rm K}( {\iota_2}\circ {\psi}_0)}
& {\,\,(\underline{\mathrm{K}}(E_2),[1_{E_2}]) \,\,} \ar[r]_-{\underline{\rm K}(\pi_2)}
& {\,\,(\underline{\mathrm{K}}(A_2),[1_{A_2}]) \,\,} \ar[r]_-{}
& {\,\,0\,\,}.}
$$
 Consider the extensions
$$0\to B_1 \xrightarrow{\iota_1} E_1\xrightarrow{\psi_1\circ \pi_1} A_2\to 0$$and$$0\to B_1 \xrightarrow{\iota_2\circ \psi_0} E_2\xrightarrow{\pi_2} A_2\to 0.$$
From the assumption and Theorem \ref{4 dengjia}, we have $E_1\cong E_2$.

\end{proof}

\begin{corollary}
Suppose $\mathcal{R}$ is a $\mathcal{DG}$-class with $\mathbb{C}\in\mathcal{R}$. For stable $C^*$-algebras $B_1,B_2\in\mathcal{R}$, $B_1,B_2$ are  isomorphic if and only if $\widetilde{B_1}\cong \widetilde{B_2},$
where
$\widetilde{B_i}$ is the minimal unitization of $B_i$, $i=1,2$.
\end{corollary}
\begin{proof}
The minimal unitization of a stable algebra forms a K-pure extension. By taking $A_1=A_2=\mathbb{C}$ in Corollary \ref{ab1b2}, we achieve the proof.

\end{proof}
\begin{remark}
Without the assumption of stability, the above is not true.
For example, $C_0[0,1)\ncong C_0[0,1/2)\oplus C_0(1/2,1]$, but
$$\widetilde{C_0[0,1)}\cong \widetilde{C_0[0,1/2)\oplus C_0(1/2,1]}\cong C[0,1].$$
\end{remark}
\begin{remark}\rm
If a $\mathcal{DG}$-class $\mathcal{R}$  is closed under stabilization and $\mathbb{C}\in\mathcal{R}$, denote $\mathcal{E}$ the class consisting of the ${\rm C}^*$-algebra $E$ arised from a unital essential K-pure extension
$$
0\to B \to E\to A\to 0,
$$
or
$E=A$,
where $A,B\in\mathcal{R}$, $A$ is unital simple, $B$ is stable. Then
$$\{A,\,\widetilde{G\otimes \mathcal{K}}|\,A, G\in\mathcal{R},\,A\,\,{\rm is\,\,unital\,\, simple}\}\subset\mathcal{E}.$$
Thus, for any $E_1, E_2\in \mathcal{E},$ by Theorem \ref{ab1b2}, we have $E_1\cong E_2$ if and only if $$(\underline{\mathrm{K}}(E_1),\underline{\mathrm{K}}(E_1)_+,[1_{E_1}])\cong
(\underline{\mathrm{K}}(E_2),\underline{\mathrm{K}}(E_2)_+,[1_{E_2}]).$$
\end{remark}
\section{Classification: the general case}

Now we tend to the not necessarily K-pure extensions, and we will show that the  total Cuntz semigroup classifies the class of ${\rm C}^*$-algebras arised from general extensions of K-pure algebras.
\begin{notion}\rm
  ({\bf The Cuntz semigroup of a ${\rm C}^*$-algebra} \cite{Cu, CEI}) Denote the cone of positive elements in $A$ by $A_+$. Let $a,b\in A_+$. 
  We say that $a$ is $Cuntz$ $subequivalent$ to $b$, denoted by $a\lesssim_{\rm Cu} b$, if there exists a sequence $(r_n)$ in $A$ such that $r_n^*br_n\rightarrow a$. One says that $a$ is  $Cuntz$ $equivalent$ to $b$, denoted by $a\sim_{\rm Cu} b$, if $a\lesssim_{\rm Cu} b$ and $b\lesssim_{\rm Cu} a$. If $a\sim_{\rm Cu} b$, then $a$ and $b$ generate the same ideal in $A$. The $Cuntz$ $semigroup$ of $A$ is defined as ${\rm Cu}(A)=(A\otimes\mathcal{K})_+/\sim_{\rm Cu}$. We will denote the class of $a\in (A\otimes\mathcal{K})_+$ in ${\rm Cu}(A)$ by $\langle a\rangle$. Using an isomorphism $M_2(\mathcal{K}) \cong \mathcal{K}$, we get an isomorphism $\psi: M_2(A \otimes \mathcal{K}) \rightarrow A \otimes \mathcal{K}$, write $ a\oplus b=\psi(\begin{array}{cc}a & 0 \\ 0 & b\end{array})$. Then ${\rm Cu}(A)$ is a positively ordered abelian semigroup when equipped with the addition: $\langle a\rangle+\langle b\rangle =\langle a \oplus b\rangle$,  and the relation:
  $$
   \langle a\rangle\leq\langle b\rangle \Leftrightarrow a\lesssim_{\rm Cu} b,\quad a,b\in (A\otimes\mathcal{K})_+.
  $$
\end{notion}

\begin{notion}\rm
Let $A$ be a ${\rm C}^*$-algebra and let Lat$(A)$ denote the collection of ideals in $A$, equipped with the partial order given by inclusion of ideals.
We say that $M$ is an ideal of ${\rm Cu}(A)$, if $M$ is an order-hereditary (for any $a,b\in {\rm Cu}(A)$, $a\leq b$ and $b\in M$ imply $a\in M$) submonoid which is closed under suprema of increasing sequences.
 For any ideal $I$ in $A$, then ${\rm Cu}(I)$ is an ideal of ${\rm Cu}(A)$. (See \cite [5.1]{APT} and \cite[3.1]{Ciu} for more details.) 
  The map $I\rightarrow {\rm Cu}(I)$ defines a lattice isomorphism between the lattice Lat$(A)$ of closed two-sided ideals of $A$ and the lattice ${\rm Lat}({\rm Cu}(A))$ of ideals of ${\rm Cu}(A)$.

  If $A$ is a ${\rm C}^*$-algebra of stable rank one. Denote ${\rm Lat}_f(A)$ the subset of ${\rm Lat}(A)$ consisting of all the ideals of $A$ that contains a full positive element.  Hence, any ideal in ${\rm Lat}_f(A)$ is singly-generated. For $a\in A_+$, we denote $I_a$ the ideal of $A$ generated by $a$.
Define {\rm Cu}$_f(I):=\{\langle a\rangle\in {\rm Cu}(A)\,|\, I_a=I\}$. In other words, ${\rm Cu}_f(I)$ consists of the elements of ${\rm Cu}(A)$ that are full in ${\rm Cu}(I)$. We will also use $I_{\langle a\rangle}$ to represent the ideal $I_a$ in $A$. If ${\langle a\rangle}\leq {\langle b\rangle}$, we will have $I_{\langle a\rangle}\subset I_{\langle b\rangle}$.

\end{notion}

The following definition combines Definition 3.18 and Theorem 3.19 in \cite{AL}. One can also see \cite[Definition 3.7]{AL} for an equivalent  version, in which we look at the KK-theory from the point of view of Cuntz \cite{Cu1}, the elements can be regarded as the homotopy classes of homomorphisms.

\begin{definition} \label{cu total def}\rm ({\bf Total Cuntz semigroup})
  Let $A$ be a separable ${\rm C}^*$-algebra of stable rank one.
   Define
  $$
  \underline{{\rm Cu}}(A):=\coprod_{I\in {\rm Lat}_f(A)} {\rm Cu}_f(I)\times {\rm K}_1(I)\times\bigoplus_{n=1}^{\infty} {\rm K}_* (I; \mathbb{Z}_n).
  $$


If $x,y\in {{\rm Cu}}(A)$ and $x\leq y$, denote $I_x$ and $I_y$ to be ideals in $A\otimes\mathcal{K}$ generated by $x$ and $y$, respectively. Denote $\iota_{I_xI_y}:\,I_x\to I_y$ to be the natural inclusion and $\delta_{I_xI_y}=\underline{\rm K}(\iota_{I_xI_y})=\underline{\rm K}( I_x)\rightarrow \underline{\rm K}( I_y)$.

  We equip $\underline{{\rm Cu}}(A)$ with addition and order as follows:
  For any
  $$\textstyle
  (x,\mathfrak{u},\bigoplus\limits_{n=1}^{\infty}(\mathfrak{s}_{n,0},\mathfrak{s}_{n,1}))\in{\rm Cu}_f(I_x)\times {\rm K}_1(I_x)\times\bigoplus\limits_{n=1}^{\infty} {\rm K}_* (I_x; \mathbb{Z}_n)
  $$
  and
$$\textstyle
  (y,\mathfrak{v},\bigoplus\limits_{n=1}^{\infty}(\mathfrak{t}_{n,0},\mathfrak{t}_{n,1}))\in{\rm Cu}_f(I_y)\times {\rm K}_1( I_y)\times\bigoplus\limits_{n=1}^{\infty} {\rm K}_* (I_y; \mathbb{Z}_n),
  $$
then
$$\textstyle
(x,\mathfrak{u},\bigoplus\limits_{n=1}^{\infty}(\mathfrak{s}_{n,0},\mathfrak{s}_{n,1}))
+(y,\mathfrak{v},\bigoplus\limits_{n=1}^{\infty}(\mathfrak{t}_{n,0},\mathfrak{t}_{n,1}))
$$
$$\textstyle
=(x+y,\delta_{I_xI_{x+y}}(\mathfrak{u},\bigoplus\limits_{n=1}^{\infty}(\mathfrak{s}_{n,0},\mathfrak{s}_{n,1}))
+\delta_{I_yI_{x+y}}(\mathfrak{v},\bigoplus\limits_{n=1}^{\infty}(\mathfrak{t}_{n,0},\mathfrak{t}_{n,1})))
$$
and
$$\textstyle
(x,\mathfrak{u},\bigoplus\limits_{n=1}^{\infty}(\mathfrak{s}_{n,0},\mathfrak{s}_{n,1}))
\leq(y,\mathfrak{v},\bigoplus\limits_{n=1}^{\infty}(\mathfrak{t}_{n,0},\mathfrak{t}_{n,1})),
$$
if
$$\textstyle
\,\,x\leq y\,\,{\rm and}\,\,\delta_{I_xI_y}
(\mathfrak{u},\bigoplus\limits_{n=1}^{\infty}(\mathfrak{s}_{n,0},\mathfrak{s}_{n,1}))=(\mathfrak{v},
\bigoplus\limits_{n=1}^{\infty}(\mathfrak{t}_{n,0},\mathfrak{t}_{n,1})),
$$

 If $A$ is unital, we set
$$\textstyle
\underline{{\rm Cu}}_u(A)=({\rm \underline{Cu}}(A),([1_A],0,\bigoplus\limits_{n=1}^{\infty}(0,0))).
$$
\end{definition}
\begin{remark}
  Let $A$ and $B$ be separable ${\rm C}^*$-algebras of stable rank one and let $\psi:\, A\to B$ be a $*$-homomorphism. Still denote $\psi\otimes {\rm id}_\mathcal{K}$ by $\psi$.
Let $x\in {\rm {Cu}}(A)$ and $y\in {\rm {Cu}}(B)$, suppose that $\psi(I_x)\subset I_y$, we denote the restriction map $\psi|_{I_x\to I_y}:\,I_x\to I_y$ by $\psi_{I_x I_y}$.

Then $\psi$ induces the two morphisms:
$$
{\rm \underline{K}}(\psi):\,{\rm \underline{K}}(A)\to{\rm \underline{K}}(B)
\quad{\rm and}
\quad {\rm \underline{K}}(\psi_{I_x I_y}):\,{\rm \underline{K}}(I_x)\to{\rm \underline{K}}(I_y).$$

In general, ${\rm \underline{K}}(\psi_{I_x I_y})$ may  not be the restriction of ${\rm \underline{K}}(\psi)$, when $I_x$ is not K-pure in $A$, ${\rm \underline{K}}(I_x)$ is not a subgroup of ${\rm \underline{K}}(A)$.

In particular,
the map ${\rm \underline{Cu}}(\psi):\,{\rm \underline{Cu}}(A)\to {\rm \underline{Cu}}(B)$ is formulated by
$$\textstyle
{\rm \underline{Cu}}(\psi)(x,\mathfrak{u},\bigoplus\limits_{n=1}^{\infty}
(\mathfrak{s}_{n,0},\mathfrak{s}_{n,1}))=({\rm {Cu}}(\psi)(x),{\rm \underline{K}}(\psi_{I_x I_{{\rm {Cu}}(\psi)(x)}})(\mathfrak{u},
\bigoplus\limits_{n=1}^{\infty}(\mathfrak{s}_{n,0},\mathfrak{s}_{n,1}))).
$$
Note that $(\mathfrak{u},\bigoplus_{n=1}^{\infty}(\mathfrak{s}_{n,0},\mathfrak{s}_{n,1}))$ is identified with $(0,\mathfrak{u},\bigoplus_{n=1}^{\infty}(\mathfrak{s}_{n,0},\mathfrak{s}_{n,1}))\in {\rm \underline{K}}(I_x)$.

Then we have
$$\textstyle
{\rm \underline{Cu}}(\psi)(x,\mathfrak{u},\bigoplus\limits_{n=1}^{\infty}(\mathfrak{s}_{n,0},\mathfrak{s}_{n,1}))\leq (y,\mathfrak{v},\bigoplus\limits_{n=1}^{\infty}(\mathfrak{t}_{n,0},\mathfrak{t}_{n,1}))$$
if and only if
$
{\rm {Cu}}(\psi)(x)\leq y\in {\rm {Cu}}(B)$ {\rm and}
$$\textstyle
{\rm \underline{K}}(\psi_{I_x I_y})(0,\mathfrak{u},\bigoplus\limits_{n=1}^{\infty}(\mathfrak{s}_{n,0},\mathfrak{s}_{n,1}))=
(0,\mathfrak{v},\bigoplus\limits_{n=1}^{\infty}(\mathfrak{t}_{n,0},\mathfrak{t}_{n,1}))\in {\rm \underline{K}}(I_y).
$$

\end{remark}
\begin{notion}\rm\label{Cuaxiom}
  ({\bf The category $\mathrm{Cu}^\sim$}) (\cite[Definition 3.7]{L1})  Let $(S, \leq)$ be an ordered monoid (not necessarily positively ordered, i.e., for any $x\in S$, we don't require $0\leq x$) such that the suprema of increasing sequences always exists in $S$. For $x$ and $y$ in $S$, let us say that $x$ is compactly contained in $y$, and denote it by  $x \ll y$, if for every increasing sequence $(y_n)$ that has a supremum  in $S$  such that $y\leq\sup _{n \geq 1} y_{n}$, then there exists $k$ such that $x\leq y_{k} .$ 
   If $x\in S$ satisfies $x\ll x$, we say that $x$ is compact.

We say that $S$ is a $\mathrm{Cu}^\sim$-semigroup of the  category $\mathrm{Cu}^\sim$, if it has a 0 element with $0\ll 0$ and satisfies the following order-theoretic axioms:

(O1): Every increasing sequence of elements in $S$ has a supremum.

(O2): For any $x \in S$, there exists a $\ll$-increasing sequence $\left(x_{n}\right)$ in $S$ such that $\sup_{n \geq 1} x_{n}=x$.

(O3): Addition and the compact containment relation are compatible.

(O4): Addition and suprema of increasing sequences are compatible.

A $\mathrm{Cu}^\sim$-morphism between two $\mathrm{Cu}^\sim$-semigroups is an ordered monoid morphism that preserves the compact containment relation and suprema of increasing sequences.
In particular, the well-known category Cu is the full subcategory of  $\mathrm{Cu}^\sim$
consisting of all those positively ordered objects.

\end{notion}

For any ${\rm Cu}^\sim$-semigroup $S$, we denote by $S_c$ the set of compact elements in $S$ and  denote ${\rm Gr}(S_c)$ the Grothendieck construction of $S_c$. Let $\phi$ be a ${\rm Cu}^\sim$-morphism between ${\rm Cu}^\sim$-semigroups, denote the restriction map of $\phi$ on the set of compact elements by $\phi_c$ and denote the induced Grothendieck map of $\phi_c$ by ${\rm Gr}(\phi_c)$.
\begin{definition} \rm {\rm (}\cite[Definition 3.15]{AL}{\rm )}
The total Cuntz category ${\rm \underline{Cu}}$ is defined as follows:
$$
{\rm Ob}({\rm \underline{Cu}})=\{\,S\in {\rm Ob}({\rm Cu}^\sim)\,|\,{\rm Gr}(S_c)\,\,{\rm is\,\, a\,\, \Lambda-module}\,\};
$$
let $X,Y\in {\rm Ob}({\rm \underline{Cu}})$, we say the map $\phi:\,X\to Y$ is a ${\rm \underline{Cu}}$-morphism if $\phi$ satisfies the following two conditions:

(1) $\phi$ is a ${\rm {Cu}}^\sim$-morphism.

(2) The induced Grothendieck map ${\rm Gr}(\phi_c): {\rm Gr}(X_c) \rightarrow {\rm Gr}(Y_c)$ is $\Lambda$-linear.

\end{definition}

The following results show that we can recover the total K-theory from the total Cuntz semigroup and under the conditions K-pure and real rank zero,
these two invariants determine each other.

\begin{proposition}{\rm(}{\rm  \cite[Proposition 5.3]{AL}}{\rm )}\label{recover prop}
  The assignment
\begin{eqnarray*}
  \qquad\qquad\underline{H}:\,  {\rm \underline{Cu}}_u\text{--}\, {\rm category} &\rightarrow& \Lambda_u\text{--}\, {\rm category} \\
  (S,u) &\mapsto & ({\rm Gr}(S_c),\rho(S_c),\rho((S_+)_c),\rho(u))\\
  \phi &\mapsto & {\rm Gr}(\phi_c)
\end{eqnarray*}
is a functor, where $\rho:\, S_c\to {\rm Gr}(S_c)$ is the natural map $(\rho(x)=[(x,0)]_{\rm Gr})$, $S_+=\{x\in S|\,x\geq 0\}$.
 The functor $\underline{H}$ yields a natural equivalence $\underline{H}\circ {\rm \underline{Cu}}_u\simeq  {\rm \underline{K}}$, which means, for any unital separable ${\rm C}^*$-algebras $A,B$ with stable rank one, if $\underline{\rm Cu}_u(A)\cong \underline{\rm Cu}_u(B)$,
then
$$
 ({\rm \underline{K}}(A),{\rm \underline{K}}(A)_+,[1_A])\cong ({\rm \underline{K}}(B),{\rm \underline{K}}(B)_+,[1_B]).
$$
\end{proposition}
   Note that if we replace ``unital" by ``has an approximate unit consisting of projections", the above statement is still true. The key point is that we can use ``large enough" projections in place of the unit in  \cite[Theorem 3.17]{AL}.

\begin{theorem}{\rm (}{\cite[Theorem 5.10]{AL}}{\rm )}\label{totalKthm}
  Upon restriction to the class of unital, separable, $\mathrm{K}$-$pure$ (or A$\mathcal{HD}$) ${\rm C}^*$-algebras of stable rank one and real rank zero, there are natural equivalences of functors:
    $$
  \underline{H}\circ \underline{\rm Cu}_{u}\simeq \underline{{\rm K}}\quad{ and}\quad   \underline{\gamma}\circ \underline{{\rm K}}\simeq \underline{\rm Cu}_{u}.
 $$
  Therefore, for these algebras, ${\rm \underline{K}}$ is a classifying functor if, and only if, so is ${\rm \underline{Cu}}_{u}$.
\end{theorem}

Consider the general extension case, we need to characterize the ideal structure of the invariant.

\begin{definition}\rm {\rm (}\cite[Section 3]{L2}{\rm )}
 Let $S$ be a $\mathrm{Cu}^{\sim}$-semigroup.  A subset $O \subseteq S$ is $Scott$-$open$ if

(i) $O$ is an upper set, that is, for any $y \in S, \,y \geq x \in O$ implies $y \in O$;

(ii) for any $x \in O$, there exists $x^{\prime} \ll x$ such that $x^{\prime} \in O$. Equivalently, for any increasing sequence of $S$ whose supremum belongs to $O$, there exists an element of the sequence also in $O$.

 We say that $S$ is $positively$ $directed$ if, for any $x \in S$, there exists $p_x \in S$ such that $x+p_x \geq 0$.

Let $S$ be a positively directed $\mathrm{Cu}^{\sim}$-semigroup and let $M$ be a submonoid of $S$. We say $M$ is $positively$ $stable$ if it satisfies the following two conditions:

(i) $M$ is a positively directed ordered monoid.

(ii) For any $x \in S$, if $\left(x+P_x\right) \cap M \neq \varnothing$, then $x \in M$, where $P_x:=\{y \in S \mid x+y \geq 0\}$.

Let $M$ be a positively stable submonoid of $S$. We say  $M$ is an
$ideal$ of $S$, if $M$ is  Scott-closed in $S$, i.e., $S\backslash M$ is scott-open.
\end{definition}
With a similar procedure of \cite[Proposition 3.16]{L2} for the unitary Cuntz semigroup,
we raise the following total version:
\begin{proposition}  Let $A$ be a ${\rm C}^*$-algebra of stable rank one. Denote
${\rm Lat}(\underline{\mathrm{Cu}}(A))$ the collection
of all the ideals of  $\underline{\mathrm{Cu}}(A)$.
Then the map
\begin{eqnarray*}
  \Phi: {\rm Lat} (A) &\to & {\rm Lat}(\underline{\mathrm{Cu}}(A)), \\
  I &\mapsto& \underline{\mathrm{Cu}}(I)
\end{eqnarray*}
is an isomorphism of complete lattices that maps ${\rm Lat}_f(A)$ onto ${\rm Lat}_f(\underline{\mathrm{Cu}}(A))$. In particular, $A$ is simple if
and only if $\underline{\mathrm{Cu}}(A)$ is simple.
\end{proposition}
\begin{proof}
By Definition \ref{cu total def}, $(x,\mathfrak{u},\bigoplus_{n=1}^{\infty}(\mathfrak{s}_{n,0},\mathfrak{s}_{n,1}))
\in\underline{\mathrm{Cu}}(A)$ belongs to $\underline{\mathrm{Cu}}(I)$ if and only if $x \in \mathrm{Cu}(I)$,  the proof is as same as \cite[Proposition
5.1.10]{APT}. For convenience, we write the inverse map as follows:
$$
\Psi : {\rm Lat}(\underline{\mathrm{Cu}}(A)) \to {\rm Lat}(A),$$
$$\textstyle
J\mapsto \{y \in A \mid(\langle yy^* \rangle, 0,\bigoplus\limits_{n=1}^{\infty}(0,\,0)) \in  J_+\}$$
where $J_+$ the positive cone of $J$.

\end{proof}
\begin{definition}\rm \label{lattice iso def}

Let $X,Y\in {\rm \underline{Cu}}$, of whom all the ideals form lattices, respectively. We say the map $\phi:\,X\to Y$ is a latticed ${\rm \underline{Cu}}$-morphism, if $\phi$ satisfies the following two conditions:

(1) $\phi$ is a ${\rm {Cu}}^\sim$-morphism.

(2) For any  ideal couples $(I_1,J_1)$, $(I_2,J_2)$ of $(X,Y)$ with $J_i$ as the ideal in $Y$ containing $\phi(I_i)$ for each $i=1,2,$ we have the following diagram commute naturally in $\Lambda$-category.
\begin{displaymath}
\xymatrixcolsep{0.4pc}
\xymatrix{
{\rm Gr}((I_1\wedge I_2)_c) \ar[rr]^-{}\ar[dr]^-{}\ar[ddd]^-{} && {\rm Gr}((I_1)_c) \ar[dr]^-{}\ar[ddd]_-{} \\
&{\rm Gr}((J_1\wedge J_2)_c) \ar[d]^-{}\ar[rr]^-{}&&
{\rm Gr}((J_1)_c) \ar[d]^-{}\\
&{\rm Gr}((J_2)_c) \ar[rr]^-{}&&
{\rm Gr}((J_1\vee J_2)_c)  \\
{\rm Gr}((I_2)_c) \ar[rr]^-{}\ar[ur]^-{} && {\rm Gr}((I_1\vee I_2)_c) \ar[ur]^-{}
}
\end{displaymath}


Particularly, if $\phi:\,X\to Y$ is a bijection with both $\phi$ and  $\phi^{-1}$ are latticed ${\rm \underline{Cu}}$-morphism, we will call $X$ and $Y$ are latticed ${\rm \underline{Cu}}$-isomorphic or $\phi$
is a latticed ${\rm \underline{Cu}}$-isomorphism, denoted by
$$
X\cong_L Y.
$$
\end{definition}
\begin{remark}\rm
It is easily seen that for any homomorphism $\psi:\,A\to B$,
where $A,B$ are separable ${\rm C}^*$-algebras of stable rank one,
${\rm \underline{Cu}}(\psi):\,{\rm \underline{Cu}}(A)\to {\rm \underline{Cu}}(B) $ is a latticed homomorphism.

Further more, suppose $I$ is an ideal of $A$. From the definition
of total Cuntz semigroup, whether $I$ is K-pure in $A$ or not, the map
${\rm \underline{Cu}}(\iota_{IA}):\,
{\rm \underline{Cu}}(I)\to {\rm \underline{Cu}}(A)$
is an injective latticed ${\rm \underline{Cu}}$-morphism,
where $\iota_{IA}:\,I\to A$ is the natural embedding map.
\end{remark}

The following is our classification theorem for the algebras obtained from general extensions.

\begin{theorem}\label{lattice Cu total classify}
Let $A_1, A_2, B_1,B_2\in\mathcal{R}$, where $\mathcal{R}$ is a $\mathcal{DG}$-class of ${\rm C}^*$-algebras. Suppose $A_1, A_2$ are unital simple and $B_1, B_2$ are stable.
Given two unital essential extensions (not necessarily {\rm K}-pure) with trivial index maps
$$
0\to B_i \xrightarrow{\iota_i} E_i\xrightarrow{\pi_i} A_i\to 0,\quad i=1,2.
$$
We have $E_1\cong E_2$ iff
$$\underline{\mathrm{Cu}}_u(E_1)\cong_L \underline{\mathrm{Cu}}_u(E_2).$$
\end{theorem}
\begin{proof}
We need only to prove the converse direction. Suppose we have $\rho:\, \underline{\mathrm{Cu}}_u(E_1)\to \underline{\mathrm{Cu}}_u(E_2)$ is a latticed $\underline{\mathrm{Cu}}_u$-isomorphism.
Denote $\rho_{|_{\rm Cu}}$ the restriction map of $\rho$ on the Cuntz semigroup: $$\rho_{|_{\rm Cu}}:\, {\mathrm{Cu}}(E_1)\to {\mathrm{Cu}}(E_2).$$
Since $E_1,E_2$ has stable rank one, the Cuntz equivalence and Murray--von Neumann equivalence coincide for projections. Thus, the restriction map of $\rho_{|_{\rm Cu}}$ on the compact positive elements is the  order-preserving  semigroup isomorphism
$$
(\rho_{|_{\rm Cu}})_c:\, \mathrm{K}_0^+(E_1)\to \mathrm{K}_0^+(E_2),
$$
whose Grothendieck map induced the  scaled order-preserving  group isomorphism
$${\rm Gr}((\rho_{|_{\rm Cu}})_c):\,(\mathrm{K}_0(E_1),\mathrm{K}_0^+(E_1),[1_{E_1}])\to (\mathrm{K}_0(E_2),\mathrm{K}_0^+(E_2),[1_{E_2}]).$$

Denote $\alpha={\rm Gr}((\rho_{|_{\rm Cu}})_c)$, applying  Lemma \ref{alpha0}, 
we have a natural commutative diagram as follows:
$$
\xymatrixcolsep{2pc}
\xymatrix{
{\,\,0\,\,} \ar[r]^-{}
& {\,\,\mathrm{K}_0(B_1)\,\,} \ar[d]_-{\alpha_0} \ar[r]^-{\mathrm{K}_0(\iota_1)}
& {\,\,\mathrm{K}_0(E_1)\,\,} \ar[d]_-{\alpha} \ar[r]^-{\mathrm{K}_0(\pi_1)}
& {\,\,\mathrm{K}_0(A_2)\,\,} \ar[d]_-{\alpha_1} \ar[r]^-{}
& {\,\,0\,\,} \\
{\,\,0\,\,} \ar[r]^-{}
& {\,\,\mathrm{K}_0(B_2)\,\,} \ar[r]_-{\mathrm{K}_0(\iota_2)}
& {\,\,\mathrm{K}_0(E_2) \,\,} \ar[r]_-{\mathrm{K}_0(\pi_2)}
& {\,\,\mathrm{K}_0(A_2) \,\,} \ar[r]_-{}
& {\,\,0\,\,},}
$$
where $\alpha_0,\alpha_1$ are also  scaled order-preserving  isomorphisms.

This means the isomorphism from ${\rm Lat}(\mathrm{K}_0(E_1))$ to ${\rm Lat}(\mathrm{K}_0(E_2))$ induced by $\alpha$
will take $\mathrm{K}_0(B_1)$ to $\mathrm{K}_0(B_2)$. Since $E_1, E_2$ are of real rank zero, the
isomorphism from ${\rm Lat}(\mathrm{Cu}(E_1))$ to ${\rm Lat}(\mathrm{Cu}(E_2))$ induced by $\rho_{|_{\rm Cu}}$
will take $\mathrm{Cu}(B_1)$ to $\mathrm{Cu}(B_2)$.
Now we claim that the latticed $\underline{\mathrm{Cu}}_u$-isomorphism $\rho$ will take $\underline{\mathrm{Cu}}(B_1)$ to $\underline{\mathrm{Cu}}(B_2)$.

For any $\textstyle
  (x,\mathfrak{u},\bigoplus\limits_{n=1}^{\infty}(\mathfrak{s}_{n,0},\mathfrak{s}_{n,1}))
  \in \underline{\mathrm{Cu}}(B_1),$
  set
  $$\textstyle
  (y,\mathfrak{v},\bigoplus\limits_{n=1}^{\infty}(\mathfrak{t}_{n,0},\mathfrak{t}_{n,1})) =:\rho(x,\mathfrak{u},\bigoplus\limits_{n=1}^{\infty}(\mathfrak{s}_{n,0},\mathfrak{s}_{n,1}))
   \in \underline{\mathrm{Cu}}(E_2);$$
$$\textstyle
(z,\mathfrak{w},\bigoplus\limits_{n=1}^{\infty}(\mathfrak{r}_{n,0},\mathfrak{r}_{n,1})) =:\rho(x,-\mathfrak{u},\bigoplus\limits_{n=1}^{\infty}(-\mathfrak{s}_{n,0},-\mathfrak{s}_{n,1}))
   \in \underline{\mathrm{Cu}}(E_2).$$
As $\rho$ is an order-preserving monoid morphism, we must have
$$\textstyle
0\leq\rho(2x,0,\bigoplus\limits_{n=1}^{\infty}(0,0))
=(y,\mathfrak{v},\bigoplus\limits_{n=1}^{\infty}(\mathfrak{t}_{n,0},\mathfrak{t}_{n,1}))
+(z,\mathfrak{w},\bigoplus\limits_{n=1}^{\infty}(\mathfrak{r}_{n,0},\mathfrak{r}_{n,1}))
$$
which implies
$$\textstyle
(\rho_{|_{\rm Cu}}(2x),0,\bigoplus\limits_{n=1}^{\infty}(0,0))
=(y+z,0,\bigoplus\limits_{n=1}^{\infty}(0,0)).
$$
Since $\rho_{|_{\rm Cu}}(2x)\in {\mathrm{Cu}}(B_2)$,
we have $y,z\in {\mathrm{Cu}}(B_2)$.
Particularly, we have
$$\textstyle
(y,\mathfrak{v},\bigoplus\limits_{n=1}^{\infty}(\mathfrak{t}_{n,0},\mathfrak{t}_{n,1}))
   \in \underline{\mathrm{Cu}}(B_2).$$
In general, the restriction of $\rho$ forms a $\mathrm{Cu}^\sim$-isomorphism
$\widetilde{\rho}:\,\underline{\mathrm{Cu}}(B_1)\to \underline{\mathrm{Cu}}(B_2).$
Moreover, by Definition \ref{lattice iso def},
$\widetilde{\rho}$ is a $\underline{\mathrm{Cu}}$-isomorphism. 

From assumption, $B_1,B_2$ have an approximate unit consisting of projections, by Proposition \ref{recover prop} (see also the paragraph below it), $\widetilde{\rho}$  induces an order-preserving  $\Lambda$-isomorphism
$$
\underline{{H}}(\widetilde{\rho}):\,(\underline{\mathrm{K}}(B_1),\underline{\mathrm{K}}(B_1)_+)
\cong (\underline{\mathrm{K}}(B_2),\underline{\mathrm{K}}(B_2)_+).
$$
By the definition of $\mathcal{DG}$ class, there exists an isomorphism
$\phi_0:\,B_1\to B_2$ such that $ \underline{\mathrm{K}}(\phi_0)=\underline{{H}}(\widetilde{\rho})$. Still via Definition  \ref{lattice iso def},
we have the following commutative diagram
$$\xymatrixcolsep{2pc}
\xymatrix{
 {\,\,\underline{\mathrm{K}}(B_1)\,\,} \ar[d]_-{\underline{\mathrm{K}}(\phi_0)}\ar[r]^-{\mathrm{K}_0(\iota_1)}
& {\,\,(\underline{\mathrm{K}}(E_1),[1_{E_1}])\,\,} \ar[d]_-{\underline{H}(\rho)} \ar[r]^-{\mathrm{K}_0(\pi_1)}
& {\,\,(\underline{\mathrm{K}}(A_1),[1_{A_1}])\,\,}\\
{\,\,\underline{\mathrm{K}}(B_2)\,\,} \ar[r]_-{\mathrm{K}_0(\iota_2)}
& {\,\,(\underline{\mathrm{K}}(E_2),[1_{E_2}]) \,\,} \ar[r]_-{\mathrm{K}_0(\pi_2)}
& {\,\,(\underline{\mathrm{K}}(A_2),[1_{A_2}]) \,\,} .}
$$
Restricting on $\mathrm{K}_*$, we have the following commutative diagram with exact rows
$$
\xymatrixcolsep{2pc}
\xymatrix{
{\,\,0\,\,} \ar[r]^-{}
& {\,\,\mathrm{K}_*(B_1)\,\,} \ar[d]_-{\mathrm{K}_*(\phi_0)} \ar[r]^-{\mathrm{K}_*(\iota_1)}
& {\,\,(\mathrm{K}_*(E_1),[1_{E_1}])\,\,} \ar[d]_-{{\underline{H}(\rho)}^*} \ar[r]^-{\mathrm{K}_*(\pi_1)}
& {\,\,(\mathrm{K}_*(A_1),[1_{A_1}])\,\,} \ar[d]_-{\varrho} \ar[r]^-{}
& {\,\,0\,\,} \\
{\,\,0\,\,} \ar[r]^-{}
& {\,\,\mathrm{K}_*(B_2)\,\,} \ar[r]_-{  \mathrm{K}_*(\iota_2)}
& {\,\,(\mathrm{K}_*(E_2),[1_{E_2}]) \,\,} \ar[r]_-{\mathrm{K}_*(\pi_2)}
& {\,\,(\mathrm{K}_*(A_2),[1_{A_2}]) \,\,} \ar[r]_-{}
& {\,\,0\,\,},}
$$ 
where $\mathrm{K}_*(\phi_0)$ is induced by $\phi_0$,  ${\underline{H}(\rho)}^*$ is the restriction map of $\underline{H}(\rho)$ between $\mathrm{K}_*$-groups
and  $\varrho$ is the induced map obtained from $\mathrm{K}_*(\phi_0)$ and ${\underline{H}(\rho)}^*$.

Note that both $\mathrm{K}_*(\phi_0)$ and ${\underline{H}(\rho)}^*$ are scaled  order-preserving maps, Proposition \ref{lin inj} implies that
 $\varrho$ is also a scaled  order-preserving map.
  By Proposition \ref{simple k* to k total},
we can lift $\varrho$ to an isomorphism $\phi_1:\,A_1\to A_2$.
Now we have the following commutative diagram with exact rows:
$$
\xymatrixcolsep{2pc}
\xymatrix{
{\,\,0\,\,} \ar[r]^-{}
& {\,\,\mathrm{K}_*(B_1)\,\,} \ar[d]_-{{\rm id}} \ar[r]^-{\mathrm{K}_*(\iota_1)}
& {\,\,(\mathrm{K}_*(E_1),[1_{E_1}])\,\,} \ar[d]_-{{\underline{H}(\rho)}^*} \ar[r]^-{\mathrm{K}_*(\phi_1\circ \pi_1) }
& {\,\,(\mathrm{K}_*(A_2),[1_{A_2}])\,\,} \ar[d]_-{{\rm id}} \ar[r]^-{}
& {\,\,0\,\,\,} \\
{\,\,0\,\,} \ar[r]^-{}
& {\,\,\mathrm{K}_*(B_1)\,\,} \ar[r]_-{ \mathrm{K}_*( \iota_2\circ\phi_0)}
& {\,\,(\mathrm{K}_*(E_2),[1_{E_2}]) \,\,} \ar[r]_-{\mathrm{K}_*(\pi_2)}
& {\,\,(\mathrm{K}_*(A_2),[1_{A_2}]) \,\,} \ar[r]_-{}
& {\,\,0\,\,}.}
$$

That is,  by Theorem \ref{strong wei}, the two extensions
$$
0\to B_1 \xrightarrow{\iota_1} E_1\xrightarrow{\phi_1\circ \pi_1} A_2\to 0
$$
and
$$
0\to B_1 \xrightarrow{\iota_2\circ \phi_0} E_2\xrightarrow{\pi_2} A_2\to 0
$$
are strongly unitarily equivalent, by Proposition \ref{strong tui cong}, then $E_1\cong E_2$.

\end{proof}
Combining Theorem \ref{zhuyao fanli} and Theorem \ref{lattice Cu total classify}, it is immediate that:
\begin{corollary}
For the algebras $E_i$, $i=1,2$ constructed in \ref{counter ex}, we have
$\underline{\mathrm{Cu}}(E_1)\ncong_L\underline{\mathrm{Cu}}(E_2)$. 
\end{corollary}

\begin{remark}
Note that  $E_1,E_2$ in \ref{counter ex}, the algebra $E$ in \cite[Example 4.5]{DL0} and a large class of ${\rm C}^*$-algebras are classified by Theorem \ref{lattice Cu total classify}.  The algebras $A_{(t,z)}$$(z\neq 0)$ in \cite[Example 20]{BD}, which are not isomorphic to any inductive limit of subhomogeneous ${\rm C}^*$-algebras, are also covered by our classification theorem.
Note that all the concrete algebras we mentioned here are not K-pure, and hence, their
$\mathrm{K}_*$-groups are not weakly unperforated.
\end{remark}
\begin{remark}
On one hand, we have shown that the total Cuntz semigroup is a complete invariant for a class of
${\rm C}^*$-algebras of stable rank one and real rank zero;
on the other hand, if $A$ is a unital, simple, exact, $\mathcal{Z}$-stable ${\rm C}^*$-algebra
of stable rank one,
we can recover both ${\rm Cu}(A)$ and $(\underline{{\rm K}}(A),\underline{{\rm K}}(A)_+, [1_A])$ from
$\underline{{\rm Cu}}_u(A)$,
which means $\underline{{\rm Cu}}_u(A)$ contains all information of
$${\rm Ell}(A) = (({\rm K}_0(A), {\rm K}_0^+(A), [1_A]),{\rm K}_1(A), {\rm T}(A), r_A)$$
(we refer the readers to \cite[Remark 4.5, Corollary 6.8]{ERS} for details).

Thus, such an invariant---total Cuntz semigroup, does work well in both  simple setting and real rank zero setting.

\end{remark}

\section*{Acknowledgements}
The first author was supported by NNSF of China (No.:12101113).
The second author was supported by NNSF of China (No.:12101102).

\end{document}